\documentclass[12pt]{article}
\usepackage{amsthm,amsfonts,amssymb,amsmath, amscd, enumerate}
\usepackage[colorlinks=true]{hyperref}
\usepackage[mathscr]{eucal}
\usepackage[all]{xy}

\newcommand{\CA}{{\mathcal {A}}}
\newcommand{\CB}{{\mathcal {B}}}

\newcommand{\CE}{{\mathcal {E}}}
\newcommand{\CF}{{\mathcal {F}}}
\newcommand{\CG}{{\mathcal {G}}}
\newcommand{\CH}{{\mathcal {H}}}
\newcommand{\CI}{{\mathcal {I}}}

\newcommand{\CL}{{\mathcal {L}}}

\newcommand{\CN}{{\mathcal {N}}}
\newcommand{\CO}{{\mathcal {O}}}
\newcommand{\CP}{{\mathcal {P}}}

\newcommand{\CR }{{\mathcal {R}}}
\newcommand{\CS}{{\mathcal {S}}}

\newcommand{\CV}{{\mathcal {V}}}

\newcommand{\CX}{{\mathcal {X}}}
\newcommand{\CY}{{\mathcal {Y}}}

\renewcommand{\div}{{\mathrm{div}}}
\newcommand{\red}{{\mathrm{red}}}

\newcommand{\Gal}{{\mathrm{Gal}}}

\newcommand{\Hom}{{\mathrm{Hom}}}

\newcommand{\Lie}{{\mathrm{Lie}}}

\newcommand{\ord}{{\mathrm{ord}}}
\newcommand{\rank}{{\mathrm{rank}}}
\newcommand{\Pic}{\mathrm{Pic}}

\renewcommand{\Im}{{\mathrm{Im}}}

\newcommand{\nef}{{\mathrm{nef}}}

\font\cyr=wncyr10  \newcommand{\Sha}{\hbox{\cyr X}}

\DeclareMathOperator{\Spec}{Spec}

\newcommand{\PP}{\mathbb{P}}    
\newcommand{\QQ}{\mathbb{Q}}

\newcommand{\ZZ}{\mathbb{Z}} 
\newcommand{\FF}{\mathbb{F}} 
\newcommand{\GG}{\mathbb{G}} 
 
\newcommand{\OO}{\overline\Omega}

\newcommand{\lra}{\longrightarrow}

\newcommand{\fppf}{\mathrm{fppf}} 
\newcommand{\et}{\mathrm{et}} 
\newcommand{\zar}{\mathrm{Zar}}

\newcommand{\fg}{\mathfrak{g}}

\newtheorem{thm}{Theorem}[section]
\newtheorem{cor}[thm]{Corollary}
\newtheorem{lem}[thm]{Lemma}
\newtheorem{pro}[thm]{Proposition}
\newtheorem{conj}[thm]{Conjecture}

\theoremstyle{definition}

\theoremstyle{remark}
\newtheorem{remark}[thm]{Remark}

\begin{document}
%------------------------------------------------------

\title{Positivity of Hodge bundles of abelian varieties over some function fields}
\author{Xinyi Yuan}
\maketitle

\tableofcontents

\section{Introduction}

Given an abelian variety $A$ over the rational function field $K=k(t)$ of a finite field $k$, we prove the following results:
\begin{itemize}
\item[(1)] $A$ is isogenous to the product of a constant abelian variety over $K$ and an abelian variety over $K$ whose N\'eron model over $\PP^1_k$ has an ample Hodge bundle. 
\item[(2)] finite generation of the abelian group $A(K^{\rm per})$ if $A$ has semi-abelian reduction over $\PP^1_k$, as part of
the ``full'' Mordell--Lang conjecture for $A$ over $K$;
\item[(3)] finiteness of the abelian group $\Sha(A)[F^\infty]$, the subgroup  of elements of the Tate--Shafarevich group $\Sha(A)$ annihilated by iterations of the relative Frobenius homomorphisms, if $A$ has semi-abelian reduction over $\PP^1_k$;
\item[(4)] the Tate conjecture for all projective and smooth surfaces $X$ over finite fields with $H^1(X,\CO_X)=0$ implies the Tate conjecture for all projective and smooth surfaces over finite fields. 
\end{itemize}
Result (1) is the main theorem of this paper, which implies the other results listed above. 
Results (2) and (3) are inspired by the paper \cite{Ros1} of Damian R\"ossler; our proof of result (1) uses a quotient construction which is indepedently introduced by Damian R\"ossler in his more recent paper \cite{Ros2}.

\subsection{Positivity of Hodge bundle}

Let $S$ be a projective and smooth curve over a field $k$, and $K=k(S)$ be the function field of $S$. 
Let $A$ be an abelian variety over $K$, and $\CA$ be the N\'eron model of $A$ over $S$ (cf. \cite[\S 1.2, Def. 1]{BLR}). 
The \emph{Hodge bundle} of $A$ over $K$ (or more precisely, of $\CA$ over $S$) is defined to be the locally free $\CO_S$-module
$$
\overline\Omega_{A}
=\overline\Omega_{\CA/S}
= e^*\Omega_{\CA/S}^1,
$$
where $\Omega_{\CA/S}^1$ the relative differential sheaf, and
$e:S\to \CA$ denotes the identity section of $\CA$. 

The \emph{height} $h(A)$ of $A$, defined to be 
$\deg(\overline\Omega_{A})$, has significant applications in Diophantine geometry.
In fact, it was used by Parshin and Zarhin to treat the Mordell conjecture over function fields and the Tate conjecture for abelian varieties over function fields. 
The number field analogue, called the Faltings height, was introduced by Faltings and plays a major role in his proof of these conjectures over number fields.

By results of Moret-Bailly and Faltings--Chai, we have $h(A)\geq 0$, or equivalently the determinant line bundle $\det(\overline\Omega_{A})$ is nef over $S$. The equality 
$h(A)= 0$ holds if and only if $\CA$ is isotrivial over $S$. 
See Theorem \ref{northcott} of the current paper. However, as we will see,  the positivity of the whole vector bundle $\overline\Omega_{A}$ is more delicate (especially in positive characteristics). The goal of this paper is to study this positivity, and gives some arithmetic applications of it.
We follow Hartshorne's notion of ample vector bundles and nef vector bundles, as in \cite{Har} and \cite[Chap. 6]{Laz}. Namely, a vector bundle $\CE$ over a scheme is 
\emph{ample} (resp. \emph{nef}) if the tautological bundle $\CO(1)$ over the projective space bundle $\PP(\CE)$ is an ample (resp. nef) line bundle. 

If $k$ has characteristic 0, it is well-known that $\overline\Omega_{A}$ is nef over $S$. This is a consequence of an analytic result of Griffiths. See also \cite[Cor. 2.7]{Bos} for an algebraic proof of this fact. 

If $k$ has a positive characteristic, $\overline\Omega_{A}$ can easily fail to be nef, as shown by the example of Moret-Bailly \cite[Prop. 3.1]{MB1}.
The example is obtained as the quotient of $(E_1\times_k E_2)_K$ by a local subgroup scheme over $K$, where $E_1$ and $E_2$ are supersingular elliptic curves over $k$.  
The quotient abelian surface has a proper N\'eron model over $S$. 

To ensure the ampleness or nefness of the Hodge bundle, one needs to impose some strong conditions. 
In this direction, R\"ossler \cite[Thm. 1.2]{Ros1} proved
that $\overline\Omega_{A}$ is nef if $A$ is an ordinary abelian variety over $K$, and 
that $\overline\Omega_{A}$ is ample if moreover
there is a place of $K$ at which $A$ has good reduction with $p$-rank $0$.

In another direction, we look for positivity by varying the abelian variety in its isogeny class. The main theorem of this paper is as follows. 

\begin{thm} \label{positivity main}
Denote $K= k(t)$ for a finite field $k$, and let $A$ be an abelian variety over $K$. 
Then $A$ is isogenous to 
$B\times_{K} C_{K}$, where $C$ is an abelian variety over $k$, and $B$ is an abelian variety over $K$ whose Hodge bundle is ample over $\PP_{k}^1$.
\end{thm}

To understand the theorem, we can take advantage of the simplicity of the theory of vector bundles on the projective line. By the Birkhoff--Grothendieck theorem (cf. \cite[Thm. 1.3.1]{HL}), any nonzero vector bundle $\CE$ on $S=\PP^1_k$ (for any base field $k$) can be decomposed as 
$$
\CE\simeq\CO(d_1) \oplus 
\CO(d_2) \oplus \cdots \oplus \CO(d_r), 
$$ 
with uniquely determined integers $d_1\geq d_2 \geq \cdots \geq d_r.$
Under this decomposition, $\CE$ is \emph{ample} if and only if $d_r>0$;
$\CE$ is \emph{nef} if and only if $d_r\geq0$. 

Go back to our theorem. 
It explains that by passing to isogenous abelian varieties, the Hodge bundle becomes nef, and the non-ample part of the nef Hodge bundle actually comes from a constant abelian variety. 
We expect the theorem to hold for the function field $K$ of any curve over any field $k$, though our proof of the current case depends heavily on many special properties of $S=\PP^1_k$.

\subsection{Purely inseparable points}

For a field $K$ of characteristic $p>0$, the \emph{perfect closure} of $K$ is the union
$$K^{\rm per}=\cup_n K^{1/p^n}$$ 
in the algebraic closure of $K$. 
The first consequence of our main theorem is the following result. 

\begin{thm}  \label{ML main}
Denote $K= k(t)$ for a finite field $k$. 
Let $A$ be an abelian variety over $K$ with everywhere semi-abelian reduction over $\PP^1_k$. 
Then $A(K^{\rm per})$ is a finitely generated abelian group. 
\end{thm}

By the Lang--N\'eron theorem, which is the function field analogue of the Mordell--Weil theorem, the theorem is equivalent to the equality $A(K^{\rm per})=A(K^{1/p^n})$ for sufficiently large $n$.

For general global function field $K$, the theorem is proved by Ghoica \cite{Ghi}  for non-isotrivial elliptic curves, and by 
R\"ossler \cite[Thm. 1.1]{Ros1} assuming that the Hodge bundle $\overline\Omega_{A}$ is ample. 
By R\"ossler's result, Theorem \ref{ML main} is a consequence of Theorem \ref{positivity main}. 
In fact, it suffices to note the fact that any $k$-morphism from $\PP^1_k$ to an abelian variety $C$ over $k$ is constant; i.e., its image is a single $k$-point of $C$. 

Finally, we remark that Theorem \ref{ML main} is related to the so-called 
\emph{full} Mordell--Lang conjecture in positive characteristic. 
Recall that the Mordell--Lang conjecture, which concerns rational points of subvarieties of abelian varieties, was proved by Faltings over number fields. A positive characteristic analogue was obtained by Hrushovski. 
However, including consideration of the \emph{$p$-part}, the full Mordell--Lang conjecture in positive characteristics, formulated by Abramovich and Voloch, requires an extra result like Theorem \ref{ML main}.
We refer to \cite{Sca, GM} for more details.
We also refer to R\"ossler \cite{Ros2} for some more recent works on this subject.

\subsection{Partial finiteness of Tate--Shafarevich group}

Let $A$ be an abelian variety over a global function field $K$ of characteristic $p$. 
Recall that the Tate--Shafarevich group of $A$ is defined by
$$
\Sha(A) =\ker \left(H^1(K, A)\lra \prod_{v} H^1(K_v, A)\right),
$$
where the product is over all places $v$ of $K$.
The prestigious Tate--Shafarevich conjecture asserts that $\Sha(A)$ is finite. By the works of Artin--Tate \cite{Tat3}, Milne \cite{Mil3}, Schneider \cite{Sch}, Bauer \cite{Bau} and Kato--Trihan \cite{KT}, the Birth and Swinnerton-Dyer conjecture for $A$ is equivalent to the finiteness of $\Sha(A)[\ell^\infty]$ for some prime $\ell$ (which is allowed to be $p$). 

Denote by $F^n: A\to A^{(p^n)}$ the relative $p^n$-Frobenius morphism over $K$. 
Define
$$
\Sha(A)[F^n]= \ker(\Sha(F^n): \Sha(A)\to \Sha(A^{(p^n)}))
$$
and 
$$\Sha(A)[F^\infty]=\bigcup_{n\geq 1} \Sha(A)[F^n].$$
Both are subgroups of $\Sha(A)$.
Note that $F^n:A\to A^{(p^n)}$ is a factor of the multiplication $[p^n]:A\to A$, so $\Sha(A)[F^\infty]$ is a subgroup of $\Sha(A)[p^\infty]$. 
These definitions generalize to the function field $K$ of a curve over any field $k$ of characteristic $p>0$. 

\begin{thm} \label{finiteness main}
Let $S$ be a projective and smooth curve over a perfect field $k$ of characteristic $p>0$, and $K$ be the function field of $S$. Let $A$ be an abelian variety over $K$. 
Then the following are true:
\begin{itemize}
\item[(1)] If $S=\PP_k^1$, the abelian variety $A$ has everywhere good reduction over 
$S$, and the Hodge bundle of $A$ is nef over $S$,  
then $\Sha(A)[F^\infty]=0$.
\item[(2)] If $A$ has everywhere semiabelian reduction over 
$S$ and the Hodge bundle of $A$ is ample over $S$, then 
$\Sha(A)[F^\infty]=\Sha(A)[F^{n_0}]$ for some positive integer $n_0$. 
\end{itemize}
\end{thm}

This theorem is analogous to R\"ossler \cite[Thm. 1.1]{Ros1}, and the proof is inspired by the loc. cit..
One consequence of Theorem \ref{positivity main} and Theorem \ref{finiteness main} is the following result.  

\begin{cor}  \label{finiteness cor}
Let $S$ be a projective and smooth curve over a finite field $k$, and $K$ be the function field of $S$. Let $A$ be an abelian variety over $K$. 
Then $\Sha(A)[F^\infty]$ is finite in each of the following cases:
\begin{itemize}
\item[(1)] $A$ is an elliptic curve over $K$;  
\item[(2)] $S=\PP_k^1$ and $A$ has everywhere semi-abelian reduction over $\PP^1_k$; 
\item[(3)] $A$ is an ordinary abelian variety over $K$, and there is a place of $K$ at which $A$ has good reduction with $p$-rank $0$.
\end{itemize}
\end{cor}

In case (1), after a finite base change, $A$ has semi-abelian reduction, and the line bundle $\OO_A$ is ample unless $A$ is isotrivial. 
In case (2), by Theorem \ref{positivity main}, it is reduced to two finiteness results corresponding to the two cases of Theorem \ref{finiteness main} exactly. 
In case (3), after a finite base change, $A$ has semi-abelian reduction, and the line bundle $\OO_A$ is ample by R\"ossler \cite[Thm. 1.2]{Ros1}.
A detailed proof of the corollary will be given in \S \ref{section inseparable}. 

Finally, we remark that case (2) of the corollary naturally arises when taking the Jacobian variety of the generic fiber of a Lefschetz fiberation of a projective and smooth surface over $k$. This standard construction was initiated by Artin--Tate \cite{Tat3} to treat the equivalence between the Tate conjecture (for the surface) and the Birch and Swinnerton-Dyer conjecture (for the Jacobian variety). We refer to Theorem \ref{Brauer Sha} for a quick review of the equivalence.

\subsection{Reduction of Tate conjecture}

One version of the prestigious Tate conjecture for \emph{divisors} is as follows. 

\begin{conj}[Conjecture $T^1(X)$] \label{Tate conjecture}
Let $X$ be a projective and smooth variety over a finite field $k$ of characteristic $p$. Then for any prime $\ell\neq p$, the cycle class map 
$$
\Pic(X)\otimes_\ZZ \QQ_p \lra H^2(X_{\bar k}, \QQ_\ell(1))^{\Gal(\bar k/k)}
$$
is surjective.
\end{conj}

The Tate conjecture is confirmed in many cases. It is proved by Tate \cite{Tat1} for arbitrary products of curves and abelian varieties. If $X$ is a K3 surface and $p>2$, the conjecture is proved by the works of Nygaard \cite{Nyg}, Nygaard--Ogus \cite{NO}, Artin--Swinnerton-Dyer \cite{ASD}, Maulik \cite{Mau}, Charles \cite{Cha} and Madapusi-Pera \cite{MP}.
Moreover, by the recent work of Morrow \cite{Mor}, Conjecture $T^1(X)$ for all projective and smooth surfaces $X$ over $k$ implies Conjecture $T^1(X)$ for all projective and smooth varieties $X$ over $k$.

In this section, we have the following reduction of the Tate conjecture. 
\begin{thm} \label{reduction main}
Let $k$ be a fixed finite field. 
Conjecture $T^1(X)$ for all projective and smooth surfaces $X$ over $k$ satisfying $H^1(X,\CO_X)=0$ implies conjecture $T^1(X)$ for all projective and smooth surfaces $X$ over $k$.
\end{thm}

\subsection{Idea of proofs} \label{section idea}
 
Here we explain our proofs of the theorems. 

\subsubsection*{Positivity of Hodge bundle} 
Theorem \ref{positivity main} is the main theorem, and its proof takes up the whole \S \ref{section positivity}. The proof consists of three major steps. 

The first step is to construct an infinite chain of abelian varieties. 
Namely, if the Hodge bundle $\OO_A=\OO_{\CA/S}$ of $A$ is not ample, then the dual $\Lie(\CA/S)$ has a nonzero \emph{maximal nef subbundle}
$\Lie(\CA/S)_\nef$. We prove that it is always a $p$-Lie algebra.
Applying the Lie theory of finite and flat group schemes developed in \cite{SGA3}, we can lift $\Lie(\CA/S)_\nef$ to a finite and flat subgroup scheme $\CA[F]_\nef$ of $\CA$ of height one. 
Then we form the quotient $\CA_1'=\CA/\CA[F]_\nef$, and let $\CA_1$ be the N\'eron model of the generic fiber of $\CA_1'$. 
If the Hodge bundle of $\CA_1$ is still not ample, repeat the construction to get $\CA_2'$ and $\CA_2$. 
Keep repeating the process, we get an infinite sequence 
$$\CA,\  \CA_1',\  \CA_1,\  \CA_2',\  \CA_2,\  \CA_3',\  \CA_3,\ \cdots.$$ 

The second step is to use heights to force the sequence to be stationary in some sense. In fact, the height of the sequence is decreasing, which is a key property proved by the construction. 
As mentioned above, the heights are non-negative integers, so the sequence of the heights is eventually constant. 
This implies in particular that there is $n_0$ such that for any $n\geq n_0$, $\Lie(\CA_n/S)_\nef$ is the base change of a $p$-Lie algebra from the base $k$, and $\CA_n[F]_\nef$ is eventually the base change of a group scheme from the base $k$.
We say such group schemes over $S$ are \emph{of constant type}. 
As a consequence, the kernels of $\CA_{n_0}\to \CA_n$ as $n$ varies give a direct systems of group schemes over $S$ of constant type. 
With some argument, we can convert this direct system into a 
$p$-divisible subgroup $\CH_\infty$ of $\CA_{n_0}[p^\infty]$ of constant type. For simplicity of notations, we assume $\CA_{n_0}$ is just $\CA$ in the following.

The third step is to ``lift'' the $p$-divisible subgroup $\CH_\infty$ of $\CA[p^\infty]$ to an abelian subscheme of $\CA$ of constant type.
By passing to a finite extension of $k$, we can find a point $s\in S(k)$ such that the fiber $C=\CA_s$ is an abelian variety over $k$.
Since $\CH_\infty$ is of constant type, it is also a $p$-divisible subgroup of $C_S[p^\infty]$.
It follows that $\CA$ and $C_S$ ``share'' the same $p$-divisible subgroup$\CH_\infty$. 
This would eventually imply that $A$ has a non-trivial $(K/k)$-trace by some fundamental theorems.
In fact, $C[p^\infty]$ is semisimple (up to isogeny) as the $p$-adic version of Tate's isogeny theorem, and thus $\CH_{\infty,K}$ is a direct summand of $C[p^\infty]$ up to isogeny. This implies that $\Hom(C_K[p^\infty], A[p^\infty])\neq 0$. 
By a theorem of de Jong \cite{Jon}, this implies that 
$\Hom(C_K, A)\neq 0$. 
Then $A$ has a non-trivial $(K/k)$-trace.
The proof is finished by applying the same process to the quotient of $A$ by the image of the $(K/k)$-trace map.

\subsubsection*{Partial finiteness} 

Theorem \ref{ML main} is an easy consequence of Theorem \ref{positivity main} and R\"ossler \cite[Thm. 1.1]{Ros1}, as mentioned above.
Theorem \ref{finiteness main} will be proved in \S \ref{section partial}. The proof is inspired by that of R\"ossler \cite[Thm. 1.1]{Ros1}, which is in turn derived from an idea of Kim \cite{Kim}.

To illustrate the idea, we first assume that $A$ is an elliptic curve with semi-abelian reduction over $S$.
Take an element $X\in \Sha(A)[F^\infty]$, viewed as an $A$-torsor over $K$.
Take a closed point $P\in X$ which is purely inseparable over $K$. It exists because $X$ is annihilated by a power of the relative Frobenius. 
Denote by $p^n$ the degree of the structure map $\psi_K: P\to \Spec K$.
Assume that $n\geq 1$.
It suffices to bound $n$ in terms of $A$.

Consider the canonical composition
$$\psi_K^* \OO_{A}\lra \Omega_{X/K}^1|_P \lra \Omega_{P/K}^1 \lra \Omega_{P/k}^1.$$ 
The first map is induced by the torsor isomorphism $X\times_KP\to A\times_KP$, and it is an isomorphism. 
The second map is surjective. 
The third map is bijective since $P$ is purely inseparable of degree $p^n$ over $K$.
We are going to extend the maps to integral models. 

Denote by $\CP$ the unique projective and smooth curve over $k$ with generic point $P$, and $\psi:\CP\to S$ be the natural map derived from $\psi_K$. 
Abstractly $\CP$ is isomorphic to $S$ since $\psi$ is purely inseparable.
By considering the minimal regular projective models of $X$ and $A$ over $S$, one can prove that
the above composition extends to a morphism
$$
\psi^* \OO_{\CA/S} \lra  \Omega_{\CP/k}^1(E).
$$
Here $\CA$ is the N\'eron model of $A$ over $S$, $E$ is the reduced structure of $\psi^{-1}(E_0)$, and $E_0$ is the set of closed points of $S$ at which $A$ has bad reduction.  
The morphism is a nonzero morphism of line bundles over $\CP$, so it is necessarily injective. The degrees on $\CP$ give
$$
p^n\cdot \deg(\OO_{\CA/S})=\deg(\psi^* \OO_{\CA/S}) \leq \deg(\Omega_{\CP/k}^1(E))=\deg(E_0)+2g-2.
$$
Here $g$ is the genus of $S$.
If $\deg(\OO_{\CA/S})> 0$, then $n$ is bounded. 
It proves the theorem in this case. 

The proof for general dimensions is based on the above strategy with 
two new ingredients. 
First, there is no minimal regular model for $A$. 
The solution is to use the compactification of Faltings--Chai \cite{FC}. 
This is the major technical part of the proof. 
Second, the Hodge bundle is a vector bundle, and we require the ampleness of the whole vector bundle.

\subsubsection*{Reduction of Tate conjecture} 

Theorem \ref{reduction main} will be proved in \S \ref{section reduction}.
One key idea is to repeatedly apply the Artin--Tate theorem, which asserts that for a reasonable fibered surface $\pi:X\to S$, the Tate conjecture $T^1(X)$ is equivalent to the BSD conjecture for the Jacobian variety $J$ of the generic fiber of $\pi$. By this, we can switch between projective and smooth surfaces over finite fields and abelian varieties over global function fields. 

As we can see from \S \ref{subsection reduction}, the major part of the proof consists of 4 steps. We describe them briefly in the following.

\medskip

\emph{Step 1: Make a fibration}.
Take a Lefschetz pencil over $X$, whose existence (over a finite base field) is proved by Nguyen \cite{Ngu}.
By blowing-up $X$, we get a Lefschetz fibration $\pi:X'\to S$ with $S=\PP^1_k$.
Denote by $J$ the Jacobian variety of the generic fiber of $\pi:X'\to S$, which is an abelian variety over $K=k(t)$ with everywhere semi-abelian reduction over $S$. 
In particular, $T^1(X)$ is equivalent to $T^1(X')$, and $T^1(X')$ is equivalent to $BSD(J)$. 

\medskip

\emph{Step 2: Make the Hodge bundle positive.}
Apply Theorem \ref{positivity main} to $J$. 
Then $J$ is isogenous to 
$A\times_{K} C_{K}$, where $C$ is an abelian variety over $k$, and $A$ is an abelian variety over $K$ with an ample Hodge bundle over 
$S$. It is easy to check that $BSD(C_K)$ holds unconditionally. 
Therefore, $BSD(J)$ is equivalent to $BSD(A)$.

\medskip

\emph{Step 3: Take a projective regular model}.
We need nice projective integral models of abelian varieties over global function fields.
This is solved by the powerful theory of Mumford \cite{Mum1} and 
Faltings--Chai \cite{FC} with some refinement by K\"unnemann \cite{Kun}.
As a result, there is a projective, flat and regular integral model $\psi:\CP\to S$ of $A^\vee\to\Spec K$ with a canonical isomorphism  
$
R^1\psi_* \CO_\CP\to \OO_A^\vee.
$
This forces $H^0(S, R^1\psi_* \CO_\CP)=0$ by the ampleness of 
$\OO_A$. 
By the Leray spectral sequence, we have 
$H^1(\CP, \CO_\CP)=0.$
This is the very reason why the positivity of the Hodge bundle is related to the vanishing of $H^1$. 

\medskip

\emph{Step 4: Take a surface in the regular model}.
By successively applying the Bertini-type theorem of Poonen \cite{Poo}, we can find a projective and smooth $k$-surface $\CY$ in $\CP$ satisfying the following conditions:
\begin{enumerate}[(1)]
\item $H^1(\CY, \CO_\CY)=0$.
\item The canonical map 
$H^1(\CP_\eta, \CO_{\CP_\eta})\to H^1(\CY_\eta, \CO_{\CY_\eta})$ is injective.
\item The generic fiber $\CY_\eta$ of $\CY\to S$ is smooth. 
\end{enumerate}
Here $\eta=\Spec K$ denotes the generic point of $S$.
Denote by $B$ the Jacobian variety of $\CY_\eta$ over $\eta$. 
Consider the homomorphism 
$A\to B$ induced by the natural homomorphism 
$\underline{\Pic}_{\CP_\eta/\eta}\to \underline{\Pic}_{\CY_\eta/\eta}$.
The kernel of $A\to B$ is finite by (2). 
It follows that $BSD(A)$ is implied by
$BSD(B)$.  
By the Artin--Tate theorem again, $BSD(B)$ is equivalent to $T^1(\CY)$. 
Note that
$H^1(\CY, \CO_\CY)=0$.
This finishes the proof of Theorem \ref{reduction main}.

\subsection{Notation and terminology}

For any field $k$, denote by $k^s$ (resp. $\bar k$) the algebraic closure (resp. separable closure). 

By a \emph{variety} over a field, we mean an integral scheme, geometrically integral, separated and of finite type over the field. By a \emph{surface} (resp. \emph{curve}), we mean a variety of dimension two (resp. one). 

We use the following basic notations:
\begin{itemize}
\item $k$ denotes a field of characteristic $p$.
\item $S$ usually denotes a projective, smooth and geometrically integral curve over $k$, which is often $\PP^1_k$.
\item $K=k(S)$ usually denotes the function field of $S$, which is often $k(t)$.  
\item $\eta=\Spec K$ denotes the generic point of $S$.
\end{itemize}
Occasionally we allow $K$ and $S$ to be more general.

\subsubsection*{Frobenius morphisms}

Let $X$ be a scheme over $\FF_p$. 
Denote by $F_X^n: X\to X$ the absolute Frobenius morphism whose induced map on the structure sheaves is given by $a\mapsto a^{p^n}$. 
To avoid confusion, we often write $F_X^n: X\to X$ as $F_X^n: X_n\to X$, so $X_n$ is just a notation for $X$. 
We also write $F^n=F_X^n$ if no confusion will result.

Let $\pi:X\to S$ be a morphism of schemes over $\FF_p$. 
Denote by
$$X^{(p^n)} =X\times_SS=(X,\pi)\times_{S} (S, F_{S}^n),$$
the fiber product of $\pi:X\to S$ with the absolute Frobenius morphism $F_{S}^n: S\to S$. 
Then $X^{(p^n)}$ is viewed as a scheme over $S$ by the projection to the second factor, and the universal property of fiber products gives an $S$-morphism 
$$F_{/S}^n=F_{X/S}^n: X\lra X^{(p^n)},$$
which is the relative $p^n$-Frobenius morphism of $X$ over $S$.
See the following diagram. 
$$
\xymatrix{
X  \ar@/_1pc/[rrdd]_{\pi}  \ar@{-->}[rrd]_{F_{X/S}^n} \ar@/^1pc/[rrrd]^{F_X^n} 
& &
\\
&
&
\quad X^{(p^n)}\quad  \ar[d] \ar[r]
&
\quad X\quad \ar[d]^{\pi}
\\
&
&
\quad S \quad \ar[r]_{F_S^n}
&
\quad S\quad
}
$$
We sometimes also write $F^n$ for $F_{X/S}^n$ if there is no confusion.

\subsubsection*{Relative Tate--Shafarevich group}
Let $K$ be a global function field, let $f:A\to B$ be a homomorphism of abelian varieties over $K$. 
Denote
$$
\Sha(A)[f]= \ker(\Sha(f): \Sha(A)\to \Sha(B)).
$$
In the case of the relative Frobenius morphism, 
$$
\Sha(A)[F^n]= \ker(\Sha(F^n): \Sha(A)\to \Sha(A^{(p^n)})).
$$
Denote 
$$\Sha(A)[F^\infty]=\bigcup_{n\geq 1} \Sha(A)[F^n]$$
as a subgroup of $\Sha(A)$.

In the setting of $f:A\to B$, we also denote
$$
A[f]=\ker(f:A\to B),
$$
viewed as a group scheme over $K$. It is often non-reduced in this paper.

\subsubsection*{Radiciel morphisms}

By \cite[I, \S 3.5]{EGA}, a morphism $f:X\to Y$ of schemes is called \emph{radicial} if one of the following equivalent conditions holds:
\begin{enumerate}[(1)]
\item the induced map $X(L)\to Y(L)$ is injective for any field $L$.
\item $f$ is universally injective; i.e., any base change of $f$ is injective on the underlying topological spaces. 
\item $f$ is injective on the underlying topological spaces, and for any $x\in X$, the induced extension $k(x)/k(f(x))$ of the residue fields is purely inseparable.
\end{enumerate}
Such properties are stable under compositions, products and base changes.

\subsubsection*{Vector bundles}

By a vector bundle on a scheme, we mean a locally free sheaf of finite rank. By a line bundle on a scheme, we mean a locally free sheaf of rank one.

\subsubsection*{Cohomology}

Most cohomologies in this paper are \'etale cohomology, if there are no specific explanations. 
We may move between different cohomology theories, and the situation will be explained time to time.

\subsubsection*{Acknowledgment}

The author is indebted to Brian Conrad for verifying some proofs in an early version of the paper. The author would also like to thank Klaus K\"unnemann, Damian R\"ossler, Yichao Tian, and Shou-Wu Zhang for important communications. 

Some part of the paper was written when the author visited China in the winter of 2017. The author would like to thank the hospitality of the Morningside Center of Mathematics, Peking University and Wuhan University. 

The author is supported by two grants (DMS-1601943, RTG/DMS-1646385) from the National Science Foundation of the USA.

\section{Positivity of Hodge bundle} \label{section positivity}

The goal of this section is to prove Theorem \ref{positivity main}.
As sketched in \S \ref{section idea}, the proof consists of three major steps. 
Each of these steps takes a subsection in \S \ref{section quotient}, 
\S \ref{section height} and \S \ref{section divisible}.
Before them, we introduce some basic results about group schemes of constant types in \S \ref{section constant}.

\subsection{Group schemes of constant type} \label{section constant}

Here we collect some basic results about group schemes to be used later. 

\subsubsection*{$p$-Lie algebras and group schemes}

Here we recall the infinitesimal Lie theory of \cite[VII$_A$]{SGA3}. 
For simplicity, we only restrict to the commutative case here. 
Let $S$ be a noetherian scheme over $\FF_p$. 
Recall that a \emph{commutative $p$-Lie algebra} over $S$ is a coherent sheaf $\mathfrak g$ on $S$, endowed with an additive morphism
$$\mathfrak g \lra \mathfrak g,\quad \delta\longmapsto \delta^{[p]}$$ 
which is \emph{$p$-linear} in the sense that 
$$(a\delta)^{[p]}=a^p\delta^{[p]}, \quad
a\in \CO_S,\ \delta\in \mathfrak g. $$ 
The additive morphism is called \emph{the $p$-th power map on $\mathfrak g$}.
We say that $\fg$ is \emph{locally free} if it is locally free as an $\CO_S$-module. 

We can interpret the $p$-th power map on $\fg$ as an $\CO_S$-linear map as follows. Recall the absolute Frobenius morphism
$F_S: S\to S$. The pull-back $F_S^*\fg$ 
is still a vector bundle on $S$.
The additive map 
$$
F_S^*: \fg\lra F_S^*\fg
$$  
is $p$-linear in that $F_S^*(a\delta)=a^pF_S^*\delta.$
It follows that we have a well-defined $\CO_S$-linear map given by 
$$
F_S^*\fg \lra  \fg, \quad 
F_S^*\delta \longmapsto \delta^{[p]}.
$$

For a commutative group scheme $\CG$ over $S$, the $\CO_S$-module
$\Lie(\CG/S)$ of invariant derivations on $\CG$ is a natural commutative $p$-Lie algebra over $S$. 
By \cite[VII$_A$, Thm. 7.2, Thm. 7.4, Rem. 7.5]{SGA3}, the functor $\CG\mapsto \Lie(\CG/S)$ is an equivalence between the following two categories:
\begin{itemize}
\item[(1)] the category of finite and flat commutative group schemes of height one over $S$,
\item[(2)] the category of locally free commutative $p$-Lie algebras over $S$. 
\end{itemize}
Here a group scheme $\CG$ over $S$ is of \emph{height one} if the relative Frobenius morphism $F_{\CG/S}:\CG\to \CG^{(p)}$ is zero. 
Furthermore, if $\CG$ is in the first category,  
then $\OO_{\CG/S}=e^*\Omega_{\CG/S}$ and $\Lie(\CG/S)$ are locally free and canonically dual to each other. 
Here $e:S\to \CG$ is the identity section.
See \cite[VII$_A$, Prop. 5.5.3]{SGA3}.

For some treatments in special cases, see \cite[\S15]{Mum2} for the case that $S$ is the spectrum of an algebraically closed field, and \cite[\S A.7]{CGP} for the case that $S$ is affine.

\subsubsection*{Group schemes of constant type}

The results below, except Lemma \ref{constant1}(1), also hold in characteristic zero. We restrict to positive characteristics for simplicity.

Let $S$ be a scheme over a field $k$ of characteristic $p>0$. 
A group scheme (resp. scheme, coherent sheaf, $p$-Lie algebra, $p$-divisible group) $\CG$ over $S$ is called \emph{of constant type over $S$} if it is isomorphic to the base change (resp. base change, pull-back, pull-back, base change) $G_S$ by $S\to \Spec k$ of some group scheme 
(resp. scheme, coherent sheaf, $p$-Lie algebra, $p$-divisible group) $G$ over $k$. 
Note that a finite flat group scheme of height one over $S$ is of constant type if and only if its $p$-Lie algebra is of constant type. 

It is also reasonable to use the term ``constant'' instead of ``of constant type'' in the above definition. However, a ``constant group scheme'' usually means a group scheme associated to an abelian group in the literature, so we choose the current terminology to avoid confusion. 

\begin{lem} \label{constant1}
Let $S$ be a Noetherian scheme over a field $k$ of characteristic $p>0$ with $\Gamma(S,\CO_S)=k$. 

\begin{itemize}
\item[(1)] Let $\pi:\CG\to S$ be a finite and flat commutative group scheme of height one over $S$. If the $p$-Lie algebra of $\CG$ is of constant type as a coherent sheaf over $S$, then $\CG$ is of constant type as a group scheme over $S$.

\item[(2)] Let $\pi:\CG\to S$ be a finite and flat commutative group scheme over $S$. If $\pi_*\CO_\CG$ is of constant type as a coherent sheaf over $S$, then $\CG$ is of constant type as a group scheme over $S$.

\item[(3)] Let $\pi_1:\CG_1\to S$ and $\pi_2:\CG_2\to S$ be finite and flat commutative group schemes of constant type over $S$. Then any $S$-homomorphism between $\CG_1$ and $\CG_2$ is of constant type, i.e., equal to the base change of a unique $k$-homomorphism between the corresponding group schemes over $k$. 
\end{itemize}
\end{lem}
\begin{proof}
We first prove (1). 
Denote $\fg=\Lie(\CG/S)$ and $\fg_0=\Gamma(S, \fg)$. 
By assumption, the canonical morphism 
$$
\fg_0 \otimes_k \CO_S \lra \fg
$$
is an isomorphism of $\CO_S$-modules. 
It suffices to prove that the $p$-th power map of $\fg$ comes from 
a $p$-th power map of $\fg_0$. 
Note that $\fg_0$ has a canonical $p$-th power map coming from global sections of $\fg$, but we do not need this fact. 

Note that the $p$-th power map of $\fg$ is equivalent to an $\CO_S$-linear map 
$F_S^*\fg \to  \fg$.
It is an element of 
\begin{multline*}
\Hom_{\CO_S}(F_S^*\fg,\fg)
=\Gamma(S, F_S^*(\fg^\vee)\otimes_{\CO_S}\fg) \\
=\Gamma(S, (F_k^*(\fg_0^\vee)\otimes_k\fg_0)\otimes_k{\CO_S})
=F_k^*(\fg_0^\vee)\otimes_k\fg_0
=\Hom_{k}(F_k^*\fg_0,\fg_0).
\end{multline*}
In other words, it is the base change of a $p$-th power map of $\fg_0$.
This proves (1).

The proof of (2) is similar. 
In fact, denote $\CF=\pi_*\CO_\CG$ and $\CF_0=\Gamma(S, \pi_*\CO_\CG)$.
The canonical morphism 
$$
\CF_0 \otimes_k \CO_S \lra \CF
$$
is an isomorphism of $\CO_S$-modules. 
Note that the structure of $\CG$ as a group scheme over $S$ is equivalent to a structure of $\CF$ as a Hopf $\CO_S$-algebra. 
For these, the extra data on $\CF$ consist of an identity map 
$\CO_S\to \CF$, a multiplication map 
$\CF\otimes_{\CO_S}\CF \to \CF$, a co-identity map $\CF\to \CO_S$, a co-multiplication map $\CF \to \CF\otimes_{\CO_S}\CF$, and an inverse map $\CF \to \CF$. 
There are many compatibility conditions on these maps. 
All these maps are $\CO_S$-linear. 
We claim that all these maps are coming from similar maps on $\CF_0$.
For example, the co-multiplication map is an element of 
\begin{multline*}
\Hom_{\CO_S}(\CF, \CF\otimes_{\CO_S}\CF)
=\Gamma(S, \CF^\vee\otimes_{\CO_S}\CF\otimes_{\CO_S}\CF) \\
=\Gamma(S, (\CF_0^\vee\otimes_k\CF_0\otimes_k\CF_0)\otimes_k{\CO_S})
=\CF_0^\vee\otimes_k\CF_0\otimes_k\CF_0
=\Hom_{k}(\CF_0, \CF_0\otimes_k\CF_0).
\end{multline*}
This makes $\CF_0$ a Hopf $k$-algebra, since the compatibility conditions hold by $\CF_0=\Gamma(S,\CF)$.   
Finally, the Hopf algebra $\CF$ is the base change of the Hopf algebra $\CF_0$. 
Then the group scheme $\CG$ is the base change of the group scheme corresponding to the 
Hopf algebra $\CF_0$. 

The proof of (3) is similar by looking at 
$\Hom_{\CO_S}((\pi_2)_*\CO_{\CG_2},(\pi_1)_*\CO_{\CG_1})$ with compatibility conditions. 
\end{proof}

The following result will be used for several times. 

\begin{lem} \label{constant2}
Denote $S=\PP_k^1$ for any field $k$ of characteristic $p>0$. 
Let 
$$
0 \lra \CG_1\lra \CG\lra \CG_2\lra 0
$$ 
be an exact sequence of finite and flat commutative group schemes over $S$. If $\CG_1$ and $\CG_2$ are of constant type, then $\CG$ is of constant type. 
\end{lem}

\begin{proof}
By Lemma \ref{constant1}(2), it suffices to prove that $\pi_*\CO_{\CG}$ is of constant type as a coherent $\CO_S$-module. Here $\pi:\CG\to S$ is the structure morphism. 
We can assume that $k$ is algebraically closed.
In fact, the property that the canonical map
$$
\Gamma(S, \pi_*\CO_{\CG}) \otimes_k \CO_S \lra \pi_*\CO_{\CG}
$$
is an isomorphism can be descended from the algebraic closure of $k$ to $k$.

Once $k$ is algebraically closed, any finite group scheme over $k$ is a successive extension of group schemes in the following list:
$$
\ZZ/\ell\ZZ, \quad
\ZZ/p\ZZ,\quad
\mu_p,\quad
\alpha_p.
$$
Here $\ell\neq p$ is any prime, and $\ZZ/\ell\ZZ$ is isomorphic to $\mu_\ell$.

For $i=1,2$, write $\CG_i=G_i\times_k S$ for a finite group scheme $G_i$ over $k$. 
By induction, we can assume that $G_1$ is one of the four group schemes over $k$ in the list.
View $\CG$ as a $G_1$-torsor over $\CG_2$. 
Then $\CG$ corresponds to a cohomology class in the fppf cohomology group $H^1_\fppf(\CG_2,G_1)$.
We first claim that the natural map 
$$
H^1_\fppf(G_2, G_1) \lra H^1_\fppf(\CG_2,G_1)
$$
is an isomorphism. 
If this holds, then $\CG$ is a trivial torsor, and thus isomorphic to 
$G_1\times_k\CG_2$ as a $\CG_2$-scheme. In particular, it is a scheme of constant type over $S$.

It remains to prove the claim that 
the natural map 
$$
H^1_\fppf(G_2, G_1) \lra H^1_\fppf(\CG_2,G_1)
$$
is an isomorphism. 
Note the basic exact sequences
$$
0\lra \alpha_p \lra \GG_a \stackrel{F}{\lra} \GG_a \lra 0,
$$
$$
0\lra \ZZ/p\ZZ \lra \GG_a \stackrel{1-F}{\lra} \GG_a \lra 0,
$$
$$
0\lra \mu_\ell \lra \GG_m \stackrel{[\ell]}{\lra} \GG_m \lra 0.
$$
In the last one, $\ell=p$ is allowed.
Then the claim is a consequence of explicit expressions on the relevant cohomology groups of $\GG_a$ and $\GG_m$ over $\CG_2$ and $G_2$.

Now we compute the cohomology groups of $\GG_a$ and $\GG_m$ over $\CG_2$ and $G_2$.
Write $R=\Gamma(G_2,\CO_{G_2})$. 
We first have 
$$
H^i_\fppf(\CG_2, \GG_a)=H^i_\zar(\CG_2, \CO_{\CG_2})
=H^i_\zar(S, R\times_k\CO_S)
=R\times_k H^i_\zar(S, \CO_S).
$$
This gives
$$
H^0_\fppf(\CG_2, \GG_a)=R, \quad H^1_\fppf(\CG_2, \GG_a)=0,\quad
H^0_\fppf(\CG_2, \GG_m)
=R^\times.
$$
To compute $H^1_\fppf(\CG_2, \GG_m)=H^1_\et(\CG_2, \GG_m)$, denote by
$I=\ker(R\to R_{\red})$ the nilradical ideal of $R$, 
and by $\CI=\ker(\CO_{\CG_2}\to \CO_{(\CG_2)_\red})$ the
the nilradical ideal sheaf of $\CG_2$. 
We have $\CI= I\otimes_k \CO_S$. 
There is an exact sequence of \'etale sheaves over $\CG_2$ by
$$
0\lra 1+\CI \lra \GG_{m,\CG_2} \lra \GG_{m,(\CG_2)_\red} \lra 0.
$$
Moreover, $1+\CI$ has a filtration 
$$
1+\CI\ \supset \ 1+\CI^2\ \supset \ 1+\CI^3 \supset \cdots,
$$
whose $m$-th quotient admits an isomorphism given by
$$
(1+\CI^m)/(1+\CI^{m+1}) \lra \CI^m/\CI^{m+1},\quad  1+t\longmapsto t.
$$
Those quotients are coherent sheaves over $\CG_2$. 
Then we have 
\begin{multline*}
H^i_\et (\CG_2, \CI^m/\CI^{m+1})
=H^i_\zar (\CG_2, \CI^m/\CI^{m+1})\\
=H^i_\zar (\CG_2, (I^m/I^{m+1})\otimes_k\CO_S)
=(I^m/I^{m+1})\otimes_k H^i_\zar(S, \CO_S).
\end{multline*}
This vanishes for $i>0$. 
As a consequence, $H^i_\et (\CG_2, 1+\CI)
=0$ for $i>0$. Therefore, 
$$
H^1(\CG_2, \GG_m)
=H^1(\CG_2, \GG_{m,(\CG_2)_\red})
=\Pic((\CG_2)_\red)
\simeq \ZZ^r,
$$
where $r$ is the number of connected components of $\CG_2$.
For the cohomology over $G_2$, similar computations give
$$
H^0_\fppf(G_2, \GG_a)=R, \ H^1_\fppf(G_2, \GG_a)=0,\
H^0_\fppf(G_2, \GG_m)=R^\times,\
H^1_\fppf(G_2, \GG_m)=0.
$$
By these, it is easy to verify the claim. 
\end{proof}

\subsection{The quotient process} \label{section quotient}

The key to the proof of Theorem \ref{positivity main} is a quotient process. 
This quotient process is also introduced by R\"ossler \cite{Ros2}.

Roughly speaking, if the Hodge bundle $\OO_A$ of $A$ is not ample, then we take the maximal nef subbundle of $\Lie(\CA/S)=\OO_A^\vee$, ``lift'' it to a local subgroup scheme of $A$, and take the quotient $A_1$ of $A$ by this subgroup scheme. 
If the Hodge bundle $\OO_{A_1}$ of $A_1$ is still not ample, then we perform the quotient process on $A_1$. Repeat this process. We obtain a sequence $A, A_1,A_2,\cdots$ of abelian varieties over $K$. 
The goal here is to introduce this quotient process. 
We start with some basic notions of vector bundles, Hodge bundles and Lie algebras.

\subsubsection*{Vector bundles on $\PP^1$}

We will introduce some basic terminologies about stability and positivity of vector bundles. 
Since we will only work over $\PP^1$, we would rather introduce them concretely instead of defining the abstract terms. 

By the Birkhoff--Grothendieck theorem (cf. \cite[Thm. 1.3.1]{HL}), any nonzero vector bundle 
$\CE$ on $S=\PP^1$ (over any base field) can be decomposed as 
$$
\CE\simeq \CO(d_1) \oplus 
\CO(d_2) \oplus \cdots \oplus \CO(d_r), 
$$ 
with uniquely determined integers $d_1\geq d_2 \geq \cdots \geq d_r.$

The \emph{slope} of $\CE$ is defined as
$$\mu(\CE)=\frac{\deg(\CE)}{\rank(\CE)} =\frac1r (d_1+ \cdots + d_r).$$
We also have the \emph{maximal slope} and the \emph{minimal slope}
$$\mu_{\max}(\CE)=d_1, \qquad \mu_{\min}(\CE)=d_r.$$
The bundle $\CE$ is \emph{semistable} if 
$\mu_{\max}(\CE)=\mu_{\min}(\CE)$. 

Under this decomposition, $\CE$ is \emph{ample} if and only if $d_r>0$;
$\CE$ is \emph{nef} if and only if $d_r\geq0$. 
Since we are mainly concerned with vector bundles on $\PP^1$, these can be served as our definitions of ampleness and nefness. 

For any nonzero vector bundle 
$\CE$ on $S=\PP^1$, define \emph{the maximal nef subbundle} (or just \emph{the nef part}) of $\CE$ to be 
$$
\CE_\nef=\Im(\Gamma(S,\CE)\otimes_k \CO_S \to \CE).
$$
In terms of the above decomposition, we simply have
$$
\CE_\nef= \oplus_{d_i\geq 0}\CO(d_i).
$$
Note that $\CE_\nef=0$ if and only $\mu_{\max}(\CE)<0$.

%A basic result asserts that if $\CE_1\to \CE_2$ is a nonzero morphisms 
%of vector bundles, then $\mu_{\min}(\CE_1)\leq \mu_{\max}(\CE_2)$.

Let $\CF$ be a coherent subsheaf of $\CE$, which is automatically a vector bundle on $S$. 
We say that $\CF$ is \emph{saturated} in $\CE$ if the quotient $\CE/\CF$ is torsion-free. Then $\CE/\CF$ is also a vector bundle on $S$. (In the literature, a saturated subsheaf is also called a \emph{subbundle}.)  Denote by $\eta$ the generic point of $S$. The functor $\CF\mapsto \CF_\eta$ is an equivalence of categories from the category of saturated subsheaves of $\CE$ to the category of linear subspaces of $\CE_\eta$. 
The inverse of the functor is $F\mapsto F\cap \CE$, an intersection 
taken in $\CE_\eta$.

\subsubsection*{Hodge bundles}

Let $\CG$ be a group scheme over a scheme $S$.
The \emph{Hodge bundle of $\CG$ over $S$} is the $\CO_S$-module
$$
\overline\Omega_{\CG}=\overline\Omega_{\CG/\CS}
= e^*\Omega_{\CG/\CS}^1,
$$
where $\Omega_{\CG/\CS}^1$ the relative differential sheaf, and
$e:S\to \CG$ denotes the identity section of $\CG$. 

Recall that if $\CG$ is a finite and flat commutative group 
scheme of height one over $S$,  
then $\overline\Omega_{\CG}$ and $\Lie(\CG/S)$ are locally free and canonically dual to each other. 
See \cite[VII$_A$, Prop. 5.5.3]{SGA3}.

The definition particularly applies to N\'eron models of abelian varieties. 
Let $S$ be a connected Dedekind scheme, and $K$ be its function field. 
Let $A$ be an abelian variety over $K$. Then we write 
$$
\OO_A=\OO_\CA=\OO_{\CA/S}. 
$$
Here $\CA$ is the N\'eron model of $A$ over $S$.

For Hodge bundles of smooth integral models of abelian varieties, we have the following well-known interpretation as the sheaf of global differentials. We sketch the idea for lack of a complete reference.

\begin{lem} \label{invariant differential}
Let $S$ be an integral scheme. 
Let $\pi:\CA\to S$ be a smooth connected group scheme whose generic fiber is an abelian variety.
There are canonical isomorphisms
$$
\pi^*\overline\Omega_{\CA/S}\lra \Omega_{\CA/S}^1, \qquad
\overline\Omega_{\CA/S}\lra \pi_*\Omega_{\CA/S}^1.
$$
\end{lem}

\begin{proof}
The first isomorphism follows from \cite[\S 4.2, Proposition 2]{BLR}. For the second map, it is well-defined using the first isomorphism. To see that it is an isomorphism, it suffices to note the following three facts:
\begin{itemize}
\item[(1)] It is an isomorphism at the generic point of $S$;
\item[(2)] Both $\overline\Omega_{\CA/S}$ and $\pi_*\Omega_{\CA/S}^1$ are torsion-free sheaves over $S$;
\item[(3)] The map is a direct summand, where a projection 
$\pi_*\Omega_{\CA/S}^1\to \overline\Omega_{\CA/S}$ is given by applying $\pi_*$ to the natural map
$\Omega_{\CA/S}^1\to e_*\overline\Omega_{\CA/S}$.
\end{itemize}
\end{proof}

\subsubsection*{Maximal nef subalgebra}

The following result gives the notion of a maximal nef $p$-Lie subalgebra of a locally free commutative $p$-Lie algebra.

\begin{lem} \label{max lie}
Let $S=\PP_k^1$ for a field $k$ of characteristic $p>0$. 
Let $\fg$ be a locally free commutative $p$-Lie algebra over $S$. 
Then the maximal nef subbundle 
$\fg_\nef$ of $\fg$ is closed under the $p$-th power map of the Lie algebra $\fg$. 
\end{lem}
\begin{proof}
Recall that the $p$-th power map on $\fg$ corresponds to an $\CO_S$-linear map $F_S^*\fg \to  \fg.$
Denote by $\CN$ the image of $F_S^*(\fg_\nef)$ under this map, which gives an $\CO_S$-linear surjection 
$F_S^*(\fg_\nef)\to \CN$.
By definition, $\fg_\nef$ is globally generated, so 
$\CN$ is also globally generated.
We have $\CN\subset \fg_\nef$ by the maximality of $\fg_\nef$. 
Then we have a well-defined $\CO_S$-linear map $F_S^*(\fg_\nef)\to \fg_\nef$.
This finishes the proof.  
\end{proof}

\subsubsection*{Quotient by subgroup scheme}

Here we describe the quotient construction in the proof of 
Theorem \ref{positivity main}. 

Go back to the setting to Theorem \ref{positivity main}. Namely, $k$ is a finite field of characteristic $p$, and $A$ is an abelian variety over $K= k(t)$. 
The $p$-Lie algebra $\Lie(\CA/S)$ of the N\'eron model $\CA$ of $A$ is a vector bundle over $S$, canonically dual to the Hodge bundle $\OO_{\CA/S}$. We also have a natural identity
$$\Lie(\CA[F]/S)=\Lie(\CA/S),$$
where $\CA[F]=\ker(F:\CA\to \CA^{(p)})$ is the kernel of the relative Frobenius morphism. 

For the sake of Theorem \ref{positivity main}, assume that $\OO_{\CA/S}$ is not ample, or equivalently $\mu_{\max}(\Lie(\CA/S)) \geq 0$. 
Then the maximal nef subbundle $\Lie(\CA/S)_\nef$ of $\Lie(\CA/S)$ is a nonzero $p$-Lie subalgebra of 
$\Lie(\CA/S)$ by Lemma \ref{max lie}. 
By the correspondence between $p$-Lie algebras and group schemes,  
$\Lie(\CA/S)_\nef$ corresponds to a finite and flat group scheme 
$\CA[F]_\nef$ of height one over $S$, which is a closed subgroup scheme of $\CA[F]$ with $p$-Lie algebra isomorphic to 
$\Lie(\CA/S)_\nef$.
Form the quotient 
$$\CA_1'=\CA/(\CA[F]_\nef),$$
which is a smooth group scheme of finite type over $S$.  
We will have a description of this quotient process in Theorem \ref{quotient}. 

Denote by $A_1$ and $A[F]_\nef$ the generic fibers of $\CA_1$ and $\CA[F]_\nef$. It follows that 
$$A_1=A/(A[F]_\nef)$$ 
is an abelian variety over $K$.
In general, $\CA_1'$ may fail to be the N\'eron model of $A_1$, or even fail to be an open subgroup scheme of the N\'eron model.
Therefore, take $\CA_1$ to be the N\'eron model of $A_1$ over $S$.

We will see that there is a natural exact sequence 
$$
0\lra \Lie(\CA/S)_\nef \lra \Lie(\CA/S) \lra \Lie(\CA_1'/S).
$$
It follows that the nef part of $\Lie(\CA/S)$ becomes zero in 
$\Lie(\CA_1'/S)$ and $\Lie(\CA_1/S)$.
However, $\Lie(\CA_1/S)$ may obtain some new nef part. 
Thus the quotient process does not solve Theorem \ref{positivity main} immediately. 
Our idea is that if $\Lie(\CA_1/S)_\nef\neq 0$, then we 
can further form the quotient 
$$\CA_2'=\CA_1/(\CA_1[F]_\nef)$$
and let $\CA_2$ be the N\'eron model of the generic fiber of $\CA_2'$.
Repeat the process, we obtain a sequence 
$$\CA=\CA_0, \ \CA_1', \ \CA_1,\  \CA_2',\  \CA_2, \ \CA_3',\ \CA_3,\ \cdots$$ 
of smooth group schemes of finite type over $S$, whose generic fibers are abelian varieties isogenous to $A$. 

To get more information from the sequence, the key is to consider the \emph{height} of the above sequence. We will see that the height sequence is decreasing. 
Since each term is a non-negative integer, the height sequence is eventually constant, and thus $\Lie(\CA_n/S)_\nef$ is eventually a direct sum of the trivial bundle $\CO_S$. 

An intermediate result that the quotient process will give us is as follows.
The proof will be given in the next subsection. 

\begin{pro} \label{positivity intermediate}
Denote $K= k(t)$ and $S=\PP_k^1$ for a field $k$ of characteristic $p>0$. 
Let $A$ be an abelian variety over $K$. 
Then there is an abelian variety $B$ over $K$, with a purely inseparable isogeny $A\to B$, satisfying one of the following two conditions:
\begin{itemize}
\item[(1)] The Hodge bundle $\OO_{B}$ is ample.
\item[(2)] The Hodge bundle $\OO_{B}$ is nef. Moreover, there is an infinite sequence $\{\CG_n\}_{n\geq 1}$
of closed subgroup schemes of the N\'eron model $\CB$ of $B$ satisfying the following conditions:
\subitem[(a)]
for any $n\geq1$,  $\CG_n$ is a finite and radicial group scheme of constant type over $S$; 
 \subitem[(b)]
for any $n\geq1$,
$\CG_n$ is a closed subgroup scheme of $\CG_{n+1}$;
\subitem[(c)]
 the order of $\CG_n$ over $S$ goes to infinity.
\end{itemize}
\end{pro}

The proposition is philosophically very similar to \cite[Prop. 2.6]{Ros2}. These two results are proved independently, but their proofs use similar ideas. For example, the ``maximal nef subalgebra'' appears in \cite[Lem. 4.7]{Ros2}, the ``quotient process'' is used in the proof of \cite[Prop. 2.6]{Ros2} in p. 20-21 of the paper, and the ``control by height'', to be introduced below by us, is also used in p. 20-21 of the paper.

\subsection{Control by heights} \label{section height}

In this subsection, we prove Proposition \ref{positivity intermediate}.
The main tool is the height of a group scheme over a projective curve. 

\subsubsection*{Heights of smooth group schemes}

Let $S$ be a projective and smooth curve over a field $k$, and let $K$ be the function field of $S$. 
Let $\CG$ be a smooth group scheme of finite type over $S$. 
The \emph{height of $\CG$} is defined to be 
$$
h(\CG)=\deg(\overline\Omega_{\CG/S})
=\deg(\det\overline\Omega_{\CG/S}).
$$
Here the Hodge bundle $\overline\Omega_{\CG/S}$ is the pull-back of the relative differential sheaf $\Omega_{\CG/S}^1$ to the identity section of $\CG$ as before.

Let $A$ be an abelian variety over $K$, and let $\CA$ be the N\'eron model of $A$ over $S$. 
The \emph{height of $A$} is defined to be 
$$
h(A)=\deg(\overline\Omega_{A})=\deg(\overline\Omega_{\CA/S})
=\deg(\det\overline\Omega_{\CA/S}).
$$
If $k$ is finite, this definition was originally used by Parshin and Zarhin to prove the Tate conjecture of abelian varieties over global function fields. A number field analogue, introduced by Faltings \cite{Fal}  and called the Faltings height, was a key ingredient in his proof of the Mordell conjecture. 

\begin{thm} \label{northcott}
Let $S$ be the projective and smooth curve over a field $k$. 
Let $\CG$ be a smooth group scheme of finite type over $S$ whose generic fiber $A$ is an abelian variety. 
Then $h(\CG)\geq h(A)\geq 0$. Moreover, the following hold:
\begin{itemize}
\item[(1)] $h(\CG)= h(A)$ if and only if $\CG$ is an open subgroup of the N\'eron model of $A$ over $S$.
\item[(2)] $h(\CG)=0$ if and only if $\CG$ is isotrivial, i.e., for some finite \'etale morphism $S'\to S$, the base change $\CG\times_SS'$ is constant over $S'$. 
\end{itemize}

\end{thm}

\begin{proof}
We first treat the inequality $h(\CG)\geq h(A)$.
Denote by $\CA$ the N\'eron model of $A$ over $S$.
By the N\'eron mapping property, there is a homomorphism $\tau:\CG\to\CA$ which is the identity map on the generic fiber. 
It induces morphisms 
$\overline\Omega_{\CA/S}\to \overline\Omega_{\CG/S}$ and $\det(\overline\Omega_{\CA/S})\to 
\det(\overline\Omega_{\CG/S})$ of locally free $\CO_S$-modules. The morphisms are isomorphisms at the generic point of $S$, and thus are injective over $S$. Taking degrees, we have $h(\CA)\leq h(\CG)$.

If $h(\CA)= h(\CG)$, then $\det(\overline\Omega_{\CA/S})\to  \det(\overline\Omega_{\CG/S})$
is an isomorphism, and thus $\overline\Omega_{\CA/S}\to \overline\Omega_{\CG/S}$ is also an isomorphism. 
By Lemma \ref{invariant differential}, the natural map 
$\tau^*\Omega_{\CA/S}\to \Omega_{\CG/S}$ is also an isomorphism. 
Consequently, $\tau:\CG\to\CA$ is \'etale. Then it is an open immersion since it is an isomorphism between the generic fibers. 

Now we check $h(A)\geq0$. 
If $\CA$ is semi-abelian, then $h(A)\geq 0$ by \cite[XI, Thm. 4.5]{MB2} or \cite[\S V.2, Prop. 2.2]{FC}. 
In general, by the semistable reduction theorem, there is a finite extension $K'$ of $K=k(S)$ such that $A_{K'}$ has everywhere semi-abelian reduction over the normalization $S'$ of $S$ in $K'$.
It follows that $h(\CA_{S'})\geq h(A_{K'})\geq 0$, and thus  
$h(\CA)\geq 0$.

If $h(\CG)=0$, the above arguments already imply that $\CG$ is an open subgroup scheme of the N\'eron model of $A$, and $A$ has everywhere semiabelian reduction. 
Now the result follows from \cite[\S V.2, Prop. 2.2]{FC}.
This finishes the proof. 
\end{proof}

\begin{remark}
If $k$ is finite, then there is a Northcott property for the height of abelian varieties, as an analogue of \cite[Chap. V, Prop. 4.6]{FC} over global function fields. 
This is the crucial property which makes the height a powerful tool in Diophantine geometry, but we do not use this property here. 
\end{remark}

\subsubsection*{Height under purely inseparable isogenies}

Here we prove a formula on the change of height under purely inseparable isogenies of smooth group schemes. 
We start with the following result about a general quotient process. 

\begin{thm} \label{quotient}
Let $\CG$ be a smooth group scheme of finite type over a Dedekind scheme $S$, and let $\CH$ be a closed subgroup scheme of $\CG$ which is flat over $S$. 
Then the fppf quotient $\CG'=\CG/\CH$ is a smooth group scheme of finite type over $S$, and the quotient morphism $\CG\to \CG'$ is faithfully flat. 
\end{thm}
\begin{proof}
The essential part follows from \cite[Thm. 4.C]{Ana}, which implies that the quotient $\CG'$ is a group scheme over $S$. 
It is easy to check that $\CG'$ is flat over $S$. 
In fact, since $\CG$ is flat over $S$, the sheaf $\CO_G$ contains no $\CO_S$-torsion. Since $\CG\to \CG'$ is an epimorphism, $\CO_{G'}$ injects into $\CO_G$, and thus contains no $\CO_S$-torsion either. 
Then $\CG'$ is flat over $S$.
To check that $\CG'$ is smooth over $S$, it suffices to check that for any geometric point $s$ of $S$, the fiber $\CG'_s$ is smooth. 
Since $\CG_s$ is reduced and $\CG_s\to \CG'_s$ is an epimorphism, $\CG'_s$ is reduced. Then $\CG'_s$ is smooth since it is a group scheme over an algebraically closed field. 
This checks that $\CG'$ is smooth over $S$. 
Moreover, $\CG_s\to \CG'_s$ is flat as $\CG_s$ is an $\CH_s$-torsor over $\CG_s'$. 
It follows that $\CG\to \CG'$ is flat by \cite[IV-3, Thm. 11.3.10]{EGA}. This finishes the proof.
\end{proof}

Now we introduce a theorem to track the change of heights of abelian varieties under the quotient process. 
The result is similar to \cite[Lem. 4.11]{Ros2}.

\begin{thm} \label{lie algebra} \label{height}
Let $S$ be a Dedekind scheme over $\FF_p$ for a prime $p$. Let $\CA$ be a smooth group scheme of finite type over $S$. 
Let $\CG$ be a closed subgroup scheme of $\CA[F]$ which is flat over $S$, and denote by $\CB=\CA/\CG$ the quotient group scheme over $S$.
Then the following hold:
\begin{itemize}
\item[(1)] There is a canonical exact sequence 
$$0 \lra \Lie(\CG/S)\lra \Lie(\CA/S)\lra \Lie(\CB/S) \lra 
F_S^*\Lie(\CG/S)\lra 0$$
of coherent sheaves over $S$. Here $F_S:S\to S$ is the absolute Frobenius morphism.  
\item[(2)]
If $S$ is a projective and smooth curve over a field $k$ of characteristic $p>0$, then
$$h(\CB)=h(\CA)-(p-1)\deg(\Lie(\CG/S)).$$ 
\end{itemize}
\end{thm}

\begin{proof}
Part (2) is a direct consequence of part (1) by $\OO_{\CA/S}=\Lie(\CA/S)^\vee$. The major problem is to prove part (1).
Consider the commutative diagram:
$$
\xymatrix{
\quad 0 \quad \ar[r] &
\quad \CG \quad \ar[r] \ar[d]^F &  \quad \CA \quad \ar[r] \ar[d]^F & \quad \CB \quad  \ar[d]^F \ar[r] & 0\\
\quad 0 \quad \ar[r] &
\quad  \CG^{(p)}\quad  \ar[r]     & \quad  \CA^{(p)}\quad  \ar[r]     & \quad \CB^{(p)} \quad \ar[r] & 0 \\
}
$$
Both rows are exact. The relative Frobenius $F:\CG\to \CG^{(p)}$ is zero as $\CG$ has height one. 
By the snake lemma, we have an exact sequence 
$$
0\lra \CG \lra \CA[F] \lra \CB[F] \lra \CG^{(p)} \lra 0.
$$
The Lie algebra of this sequence is exactly the sequence of the theorem. 

It suffices to check the general fact that the Lie functor from the category of finite and flat commutative group schemes of height one over $S$ to the category of $p$-Lie algebras 
is exact. 
In fact, if 
$$
0\lra\CH_1 \lra \CH \lra \CH_2 \lra 0
$$
is an exact sequence in the first category, then we first get a complex 
$$0 \lra \Lie(\CH_1/S)\lra \Lie(\CH/S) {\lra} \Lie(\CH_2/S)\lra 0$$
of locally free sheaves over $S$.
By the canonical duality between the Lie algebra and the Hodge bundle, 
the Lie functor commutes with base change. 
For any point $s\in S$, consider the fiber of the complex of Lie algebras above $s$. 
It is exact by counting dimensions, since $\dim \Lie(\CH_s/s)$ equals the order of $\CH_s$ and the order is additive under short exact sequences. 
This shows that the complex is fiber-wise exact.  
Then the complex is exact.
This finishes the proof.
\end{proof}

\subsubsection*{Stationary height}

Now we prove Proposition \ref{positivity intermediate}. 
Let $\CA=\CA_0$ be as in the proposition. 
By the quotient process, we obtain a sequence 
$$
\CA_0,\ \CA_1',\ \CA_1,\ \CA_2',\ \CA_2,\ \cdots
$$ 
of abelian varieties over $K$. Here for any $n\geq0$,
$$
\CA_{n+1}'=\CA_n/\CA_n[F]_\nef
$$
and $\CA_{n+1}$ is the N\'eron model of the generic fiber of $\CA_{n+1}'$. 
We know that $\CA_n'$ is a smooth group scheme of finite type over $S$ by Theorem \ref{quotient}. 

By Theorem \ref{northcott}, $h(\CA_n)\leq h(\CA_n')$. 
By Theorem \ref{height}(2), 
$$h(\CA_{n+1}')=h(\CA_n)-(p-1)\deg(\Lie(\CA_n/S)_\nef) \leq h(\CA_n).$$ 
It follows that the sequence
$$
h(\CA_0),\ h(\CA_1'),\ h(\CA_1),\ h(\CA_2'),\ h(\CA_2),\ \cdots
$$ 
is decreasing. 
Note that each term of the sequence is a non-negative integer by Theorem \ref{northcott}.
Therefore, there is an integer $n_0\geq 0$ such that 
$$h(\CA_{n}')=h(\CA_{n})=h(\CA_{n_0}), \quad \forall\ n\geq n_0.$$ 
It follows that for any $n\geq n_0$, $\CA_{n}'$ is an open subgroup scheme of $\CA_{n}$ and 
$$\deg(\Lie(\CA_n/S)_\nef)=\deg(\Lie(\CA_n'/S)_\nef)=0.$$ 
As a consequence, for any $n\geq n_0$, $\Lie(\CA_n/S)_\nef$ is a direct sum of copies of the trivial bundle $\CO_S$, and 
$\OO_{\CA_n/S}=\Lie(\CA_n/S)^\vee$ is nef. 
Moreover, $\OO_{\CA_n/S}$ is ample if and only if $\CA_n\to \CA_{n+1}$ is an isomorphism.  

For Proposition \ref{positivity intermediate}, if none of $\OO_{\CA_n/S}$ is ample, take $\CB=\CA_{n_0}$. The group scheme $\ker(\CA_{n_0}\to \CA_n)$ has a degree going to infinity.
It is of constant type over $S$, as an easy consequence of Lemma \ref{constant1}(1) and Lemma \ref{constant2}. 
This proves Proposition \ref{positivity intermediate}.

\subsection{Lifting $p$-divisible groups} \label{section divisible}

In this section, we prove Theorem \ref{positivity main}.
Note that we have already proved Proposition \ref{positivity intermediate}. 
To finish the proof, it suffices to prove the following result:

\begin{pro} \label{positivity intermediate2}
Denote $K= k(t)$ and $S=\PP_k^1$ for a finite field $k$ of characteristic $p$. 
Let $\CA$ be a smooth group scheme of finite type over $S$ whose generic fiber $A$ is an abelian variety over $K$.  
Assume that there is an infinite sequence $\{\CG_n\}_{n\geq 1}$
of closed subgroup schemes of $\CA$ satisfying the following conditions:
\begin{itemize}
\item[(a)]
for any $n\geq1$, $\CG_n$ is a finite group scheme of constant type over $S$; 
\item[(b)]
for any $n\geq1$, $\CG_n$ is a subgroup scheme of $\CG_{n+1}$; 
\item[(c)]
 the order of $\CG_n$ over $S$ is a power of $p$ and goes to infinity. 
\end{itemize}
Then the $(K/k)$-trace of $A$ is non-trivial. 
\end{pro}

We refer to \cite{Con} for the theory of $(K/k)$-traces. 
Before proving Proposition \ref{positivity intermediate2}, let us see how Proposition \ref{positivity intermediate} and 
Proposition \ref{positivity intermediate2} imply Theorem \ref{positivity main}.
Let $A$ be as in Theorem \ref{positivity main}.
Apply Proposition \ref{positivity intermediate} to $A$, which gives an abelian variety $B$ over $K$ with a purely inseparable isogeny $A\to B$.
If $B$ satisfies Proposition \ref{positivity intermediate}(1), the result already holds. 
If $B$ satisfies Proposition \ref{positivity intermediate}(2), apply 
Proposition \ref{positivity intermediate2} to the N\'eron model $\CB$ of $B$. 
Then $B$ has a non-trivial $(K/k)$-trace, and thus 
$A$ also has a non-trivial $(K/k)$-trace $A_0$, which is an abelian variety over $k$ with a homomorphism $A_{0,K}\to A$.
By \cite[Theorem 6.4]{Con}, the homomorphism $A_{0,K}\to A$ is an isogeny to its image $A'$. Note that $A$ is isogenous to 
$A_{0,K}\times_K (A/A')$. 
Apply the same process to the abelian variety $A/A'$ over $K$. 
Note that the dimension of $A/A'$ is strictly smaller than that of $A$. The process eventually terminates. 
This proves Theorem \ref{positivity main}.

\subsubsection*{The $p$-divisible group}

To prove Proposition \ref{positivity intermediate2}, the first step is to  
change the direct system $\{\CG_n\}_n$ to a nonzero $p$-divisible group. 
For the basics of $p$-divisible groups, we refer to Tate \cite{Tat2}. 

Let $S$ be any scheme. A direct system $\{G_n\}_{n\geq 1}$ of flat group schemes over $S$ is called an \emph{increasing system} if the transition homomorphisms are closed immersions. 
A \emph{subsystem} of the increasing system $\{G_n\}_{n\geq 1}$ 
is an increasing system $\{H_n\}_{n\geq 1}$ over $S$ endowed with an injection $\varinjlim H_n \to \varinjlim G_n$ as fppf sheaves over $S$. 

There is a description of subsystem in terms of group schemes without going to the limit sheaves.
In fact, an injection 
$\varinjlim H_n \to \varinjlim G_n$ as fppf sheaves over $S$ is equivalent to a sequence 
$\{\phi_n:H_n\to G_{\tau(n)}\}_{n\geq1}$ of injections, compatible with the transition maps $H_n\to H_{n+1}$ and $G_{\tau(n)}\to G_{\tau(n+1)}$ for each $n\geq 1$, where
$\{\tau(n)\}_{n\geq1}$  is an increasing sequence of positive integers.
For each $n\geq1$, to find $\tau(n)$, it suffices to note that the identity map $i_n:H_n\to H_n$ is an element of $H_n(H_n)\subset H_\infty(H_n) \subset G_\infty(H_n)$, and thus it is contained in some $G_{\tau(n)}(H_n)$. This gives a morphism $H_n\to G_{\tau(n)}$. 

We have the following basic result. 
\begin{lem}    \label{divisible}
Let $k$ be any field of characteristic $p>0$. 
Let $\{G_n\}_{n\geq 1}$ be an increasing system of finite commutative group schemes of $p$-power order over $k$. 
Assume that the order of $G_n$ goes to infinity as $n\to \infty$, but the order of $G_n[p]$ is bounded as $n\to \infty$. 
Then $\{G_n\}_{n\geq 1}$ has a subsystem $\{H_n\}_{n\geq 1}$ which is a nonzero $p$-divisible group over $k$.
\end{lem}

\begin{proof}
The idea can be easily illustrated in terms of abelian groups.
Assume for the moment that $\{G_n\}_{n\geq 1}$ is an increasing system of abelian groups satisfying similar conditions. 
Let $G_\infty$ be the direct limit of $\{G_n\}_{n\geq 1}$. 
By definition, $G_\infty$ is an infinite torsion group whose element has $p$-power orders, but $G_\infty[p]$ is finite. 
A structure theorem asserts that 
$G_\infty\simeq (\QQ_p/\ZZ_p)^r\oplus F$ for some positive integer $r$ and some finite group $F$.
Then $H_\infty=(\QQ_p/\ZZ_p)^r$ is the subgroup of $G_\infty$ that gives us a $p$-divisible group. 
This subgroup consists of exactly the infinitely divisible elements of 
$G_\infty$, and thus can be extracted as $H_\infty=\cap_{a\geq 1} p^aG_\infty$. Then $H_m=H_\infty[p^m]=\cap_{a\geq 1} (p^aG_\infty)[p^m]$ for any $m\geq1$.

Go back to the group schemes $G_n$ in the lemma. 
By assumption, the order of $G_n[p]$ is bounded by some integer $p^r$. 
For any $m$, the order of $G_n[p^m]$ is bounded by $p^{mr}$, which can be checked by induction using the exact sequence 
$$
0\lra G_n[p]\lra G_n[p^{m+1}]\stackrel{[p]}{\lra} G_n[p^{m}].
$$
Since the order of $G_n$ goes to infinity, the exact exponent of $G_n$, which is the smallest positive integer $N_n$ such that the multiplication $[N_n]:G_n\to G_n$ is the zero map, also goes to infinity. 

Now we construct the $p$-divisible group. 
Denote $G_\infty=\varinjlim G_n$ as an fppf sheaf over $\Spec(k)$. 
Denote $H_\infty=\cap_{a\geq 1} p^aG_\infty$ as a subsheaf of 
$G_\infty$.
Denote $H_m=H_\infty[p^m]$ as a subsheaf of $H_\infty$ for any $m\geq1$. 
We claim that the system $\{H_m\}_{m\geq1}$, where the transition maps are injections as subsheaves of $H_\infty$, is a nonzero $p$-divisible group over $k$. 

First, every $H_m$ is representable by a finite group scheme over $k$.
In fact, $H_m$ is the intersection of the decreasing sequence 
$\{(p^aG_\infty)[p^m]\}_{a\geq1}$.
Since the order of  $\{G_n[p^m]\}_n$ is bounded, the increasing sequence $\{(p^aG_n)[p^m]\}_n$ of finite group schemes
is eventually stationary. 
This stationary term is exactly $(p^aG_\infty)[p^m]$.
The sequence $\{(p^aG_\infty)[p^m]\}_{a\geq 1}$ of finite group schemes is decreasing, and thus eventually stationary.
This stationary term is exactly $H_m$.  

Second, $H_1\neq 0$. Otherwise, 
$(p^aG_\infty)[p]=0$ for some $a\geq 1$. Then $p^aG_\infty=0$ and thus $p^aG_n=0$ for all $n$. This contradicts to the fact that the exponent of $G_n$ goes to infinity. Thus $H_1\neq 0$. 

By definition, the map $[p^m]:H_\infty\to H_\infty$ is surjective with kernel $H_m$.
It follows that the morphism 
$[p^m]:H_{m+1}\to H_{m+1}$ has kernel $H_m$ and image $H_1$.
This implies that $\{H_m\}_{m\geq1}$ is a $p$-divisible group.
This finishes the proof.
\end{proof}

\begin{remark}
If $k$ is perfect, one can prove the lemma by Dieudonn\'e modules. 
In fact, take the covariant Dieudonn\'e module of the sequence $\{G_n\}_{n\geq 1}$, apply the above construction of abelian groups to the Dieudonn\'e modules, and transfer the result back to get a $p$-divisible group by the equivalence between finite group schemes and Dieudonn\'e modules. \end{remark}

\subsubsection*{Algebraicity}

Now we prove Proposition \ref{positivity intermediate2}.
Recall that we have an increasing system $\{\CG_n\}_n$ of finite and flat closed subgroup schemes of $\CA$ of constant type. 
The transition maps are necessarily of constant type by Lemma \ref{constant1}(3).
Thus $\{\CG_n\}_n$ is the base change of an increasing system $\{G_n\}_n$ of finite group schemes over $k$. 
By Lemma \ref{divisible}, the system $\{G_n\}_n$ has a subsystem $H_\infty=\{H_n\}_n$, which is a nonzero $p$-divisible group  over $k$.
Denote by $\CH_\infty=\{\CH_n\}_n$ the base change of $\{H_n\}_n$ to $S$, which is a $p$-divisible group over $S$, and also a subsystem of 
$\{\CG_n\}_n$. Then $\CH_\infty=\{\CH_n\}_n$ is a subsystem of 
$\CA[p^\infty]=\{\CA[p^n]\}_n$.
We are going to ``lift'' $\CH_\infty=\{\CH_n\}_n$ to an abelian scheme over $S$ of constant type. 

By \cite[Thm. 6.6]{Con}, the $(K/k)$-trace of $A_{K}$ is nonzero if and only if the $(Kk'/k')$-trace of $A_{K'}$ is nonzero for any extension $k'/k$. Therefore, in the proposition, we can replace $k$ by any finite extension. In particular, we can assume that there is a point $s\in S(k)$ such that $\CA$ has good reduction at $s$. The fiber $C=\CA_s$ is an abelian variety over $k$, and the $p$-divisible group $C[p^\infty]$ has a $p$-divisible subgroup $\CH_{\infty,s}$, which is canonically isomorphic to $H_\infty$. 
We will prove that $\Hom(C_K,A)\neq 0$ from the fact that they share the same $p$-divisible subgroup $H_{\infty,K}$.

To proceed, we need two fundamental theorems on $p$-divisible groups of abelian varieties over finitely generated fields. 

\begin{thm} \label{hom}
Let $K$ be a finitely generated field over a finite field $\FF_p$.
Let $A$ and $B$ be abelian varieties over $K$. 
Then the canonical map 
$$
\Hom(A,B)\otimes_\ZZ\ZZ_p \lra \Hom(A[p^\infty], B[p^\infty])
$$
is an isomorphism.
\end{thm}

\begin{thm} \label{semisimple}
Let $A$ be an abelian variety over a finite field $k$ of characteristic $p>0$. 
Then the $p$-divisible group $A[p^\infty]$ is semi-simple, i.e., isogenous to a direct sum of simple $p$-divisible groups over $k$. 
\end{thm}

The more classical $\ell$-adic analogues of theorems are the Tate conjectures and the semi-simplicity conjecture proved by Tate and Zarhin.
For the current $p$-adic version, Theorem \ref{hom} for a finite field $K$ and Theorem \ref{semisimple} can be proved by an easy modification of the $\ell$-adic argument of Tate \cite{Tat1}.
For general $K$, Theorem \ref{hom} is proved by de Jong \cite{Jon}.

Go back to the proof of Proposition \ref{positivity intermediate2}. 
By Theorem \ref{semisimple}, the injection $H_{\infty}\to C[p^\infty]$ implies the existence of a surjection $C[p^\infty]\to H_{\infty}$. Take a base change to $K$ and compose with $H_{\infty,K}\to A[p^\infty]$. We have a nonzero element of 
$\Hom(C_K[p^\infty], A[p^\infty])$. 
By Theorem \ref{hom}, we have
$\Hom(C_K,A)\neq 0$. 
This proves the proposition.

\section{Purely inseparable points on torsors} \label{section partial}

The goal of this section is to prove Theorem \ref{finiteness main}. 
In \S \ref{section torsor}, we review some basic results on torsors. 
In \S \ref{section inseparable}, we prove Theorem \ref{finiteness main}. 
The idea of the proof is explained on \S \ref{section idea}.

\subsection{Preliminary results on torsors} \label{section torsor}

\subsubsection*{N\'eron model of locally trivial torsor}

Let $S$ be a Dedekind scheme and $K$ be its function field. 
Let $X$ be a smooth and separated scheme of finite type over $K$. 
Recall from \cite[\S 1.2, Def. 1]{BLR} that a \emph{N\'eron model} $\CX$ of $X$ is a smooth and separated 
$S$-scheme of finite type with a $K$-isomorphism $X\to\CX_K$ satisfying the 
\emph{N\'eron mapping property} that, for any smooth $S$-scheme $\CY$ and any $K$-morphism $\CY_K\to X$, there is a unique $S$-morphism $\CY\to \CX$ extending the morphism $\CY_K\to X$. It is immediate that a N\'eron model is unique if it exists. 

The main goal of \cite{BLR} is a complete and modern proof of the statement that any abelian variety over $K$ admits a N\'eron model. Implicitly, the book contains the following result for locally trivial torsors of abelian varieties.

\begin{thm} \label{neron torsor}
Let $A$ be an abelian variety over $K$ and $X$ be an $A$-torsor over $K$. 
Assume that $X$ is trivial over the completion $K_v$ of $K$ with respect to the discrete valuation induced by any closed point $v\in S$. 
Then $X$ (resp. $A$) admits a unique N\'eron model $\CX$ (resp. $\CA$) over $S$.
Moreover, the torsor structure $A\times_KX\to X$ extends uniquely to an $S$-morphism $\CA\times_S\CX\to \CX$, which makes $\CX$ an $\CA$-torsor. 
\end{thm}

\begin{proof}
We sketch a proof for the $\CX$-part in the following. 
\begin{itemize}
\item[(1)] The local N\'eron model exists. Namely, for any closed point $v\in S$, 
the N\'eron model $\CX_{O_{S,v}}$ of $X$ over the local ring $O_{S,v}$ exists.
Moreover, $\CX_{O_{S,v}}$ is a natural $\CA_{O_{S,v}}$-torsor over $O_{S,v}$. 
It is a consequence of \cite[\S 6.5, Rem. 5]{BLR} by taking $R=O_{S,v}$ and $R'$  to be the completion of $O_{S,v}$.

\item[(2)]
The global N\'eron model $\CX$ over $S$ exists by patching the local ones. 
This follows from \cite[\S 1.4, Prop. 1]{BLR}.

\item[(3)] 
The torsor structure extends to $(\CA, \CX)$. 
By the N\'eron mapping property, the torsor structure map $A\times_KX\to X$ extends uniquely to 
a morphism $\CA \times_S \CX \to \CX$. 
To see the later gives a torsor structure, we need to verify that the induced map 
$\CA \times_S \CX \to \CX \times_S \CX$ is an isomorphism. 
This is true because it is true over $O_{S,v}$ for every $v$. 
\end{itemize}
\end{proof}

We can also define Hodge bundles of N\'eron models of torsors. 
In fact, in the setting of Theorem \ref{neron torsor}, define \emph{the Hodge bundle associated to $X$} to be 
$$
\overline\Omega_{X}
=\overline\Omega_{\CX/S}
=\pi'_*\Omega_{\CX/S}^1.
$$
Here $\pi':\CX\to S$ is the structure morphism. 
If $X=A$, this agrees with the definition of Hodge bundles of abelian varieties in \S\ref{section quotient} by viewing $A$ as an abelian variety by Lemma \ref{invariant differential}.

Similar to Lemma \ref{invariant differential}, the natural morphism
$$
\pi'^*\overline\Omega_{\CX/S}\lra \Omega_{\CX/S}^1
$$
 is an isomorphism. In fact, 
take a faithfully flat base change $S'\to S$ trivializing $\CX$. 
Then the map becomes an isomorphism after the base change, and it is an isomorphism before the base change by the flat descent. 

The following result asserts that $\overline\Omega_{X}$ is a vector bundle on $S$ which has very similar numerical property as 
$\overline\Omega_{A}$.  

\begin{lem} \label{hodge bundle}
Let $\psi:S'\to S$ be a morphism such that the $S$-torsor
$\CX$ is trivial over $S'$. 
Then there is a natural isomorphism
$$
 \psi^*\overline\Omega_{\CX/S}
\lra
 \psi^*\overline\Omega_{\CA/S}
$$
of $\CO_{S'}$-modules, depending on the choice of  an 
$S$-morphism $S'\to \CX$.  
\end{lem}
\begin{proof}
Take the base change $\psi:S'\to S$ which trivializes $\CX$. 
Denote 
$$\CX'=\CX\times_SS', \quad
\CA'=\CA\times_SS'.$$
The base change gives a canonical section $S'\hookrightarrow \CX'$ lifting $S'\to \CX$. Using this section, we can view the $\CA'$-torsor $\CX'$ as a group scheme over $S'$. 
It follows that 
$$
 \psi^*\overline\Omega_{\CX/S}
=
(\Omega_{\CX/S}^1) |_{S'}
\simeq (\Omega_{\CX'/S'}^1) |_{S'}
\simeq (\Omega_{\CA'/S'}^1) |_{S'}
\simeq  \psi^*\overline\Omega_{\CA/S}.
$$
The result follows. 
\end{proof}

\subsubsection*{Functoriality and base change}

We first present a basic result on the relative Frobenius morphism of abelian varieties. Let $A$ be an abelian variety over a field $K$ of characteristic $p$. 
Consider the following two maps:
\begin{itemize}
\item[(1)] (Functoriality map) The map 
$$H^1(F^n):H^1(K, A)\lra H^1(K, A^{(p^n)})$$ 
induced by the relative Frobenius morphism $F^n:A\to A^{(p^n)}$ via functoriality.  
It sends an $A$-torsor $X$ to the $A^{(p^n)}$-torsor $X/(A[F^n])$, where $A[F^n]$ is the kernel of $F^n:A\to A^{(p^n)}$.
\item[(2)] (Base change map) The map 
$$(F_{K}^n)^*:H^1(K, A)\lra H^1(K, A^{(p^n)})$$ 
induced by the morphism $F_{K}^n: \Spec K\to \Spec K$ of schemes, where $A^{(p^n)}$ is viewed as the pull-back of the \'etale sheaf $A$ via $F_{K}^n$.
It sends an $A$-torsor $X$ to the $A^{(p^n)}$-torsor 
$X^{(p^n)}=X\times_K (K,F_K^n)$. 
\end{itemize}

\begin{lem} \label{sha base change}
The above maps $H^1(F^n)$ and $(F_{K}^n)^*$ are equal. 
\end{lem}
\begin{proof}
We first present a geometric interpretation, which can be turned to a rigorous proof. Recall the relative Frobenius morphism $X\to X^{(p^n)}$. The action of $A$ on $X$ induces an action of $A[F^n]$ on $X$. The quotient map of the latter action is exactly $X\to X^{(p^n)}$.

We can also prove the result in terms of cocycles for Galois cohomology.
In fact, for any torsor $X\in H^1(K, A)$, take a point $P\in X(K^{\rm sep})$. 
By definition, $X$ is represented by the cocycle 
$$\Gal(K^{\rm sep}/K)\lra A(K^{\rm sep}), \quad
\sigma\longmapsto P^\sigma-P.$$ 
Then $F^n(P)$ is a point in $X^{(p^n)}$. This point gives a cocycle representing 
$X^{(p^n)}$ by
$$\sigma\longmapsto F^n(P)^\sigma-F^n(P)=F^n(P^\sigma-P)
\in A^{(p^n)}(K^{\rm sep}),$$ 
which is exactly the image of $H^1(F^n)$. 
This proves that the maps are equal. 
\end{proof}

\subsection{Purely inseparable points} \label{section inseparable}

In this subsection, we prove Theorem \ref{finiteness main} and Corollary \ref{finiteness cor}. For convenience, we duplicate 
Theorem \ref{finiteness main} in the following.

\begin{thm}[Theorem \ref{finiteness main}]
Let $S$ be a projective and smooth curve over a perfect field $k$ of characteristic $p>0$, and $K$ be the function field of $S$. Let $A$ be an abelian variety over $K$. 
Then the following are true:
\begin{itemize}
\item[(1)] If $S=\PP_k^1$, $A$ has everywhere good reduction over 
$S$, and the Hodge bundle of $A$ is nef over $S$,  
then $\Sha(A)[F^\infty]=0$.
\item[(2)] If $A$ has everywhere semiabelian reduction over 
$S$ and the Hodge bundle of $A$ is ample over $S$, then 
$\Sha(A)[F^\infty]=\Sha(A)[F^{n_0}]$ for some positive integer $n_0$. 
\end{itemize}
\end{thm}

\subsubsection*{The corollary}

Now we deduce Corollary \ref{finiteness cor} from Theorem \ref{positivity main} and Theorem \ref{finiteness main}, which is duplicated below. 
 
\begin{cor} [Corollary \ref{finiteness cor}] 
Let $S$ be a projective and smooth curve over a finite field $k$, and $K$ be the function field of $S$. Let $A$ be an abelian variety over $K$. 
Then $\Sha(A)[F^\infty]$ is finite in each of the following cases:
\begin{itemize}
\item[(1)] $A$ is an elliptic curve over $K$;  
\item[(2)] $S=\PP_k^1$ and $A$ has everywhere semi-abelian reduction over $\PP^1_k$;
\item[(3)] $A$ is an ordinary abelian variety over $K$, and there is a place of $K$ at which $A$ has good reduction with $p$-rank $0$. 
\end{itemize}
\end{cor}

We first prove (1). Let $K'$ be a finite Galois extension of $K$.
By the inflation-restriction exact sequence, we see that the kernel of 
$\Sha(A)\to \Sha(A_{K'})$ is annihilated by $[K':K]$. 
This kernel is actually finite by Milne \cite{Mil2}. 
Consequently, we can replace $K$ by any finite Galois extension, and we can particularly assume that $A$ has everywhere semi-abelian reduction over $S$. Note that $\OO_A$ is a line bundle over $S$.
The height $h(A)=\deg(\OO_A)\geq 0$, where the equality holds only if $A$ is isotrivial. 
This is a classical fact for elliptic curves, but we also refer to \cite[\S V.2, Prop. 2.2]{FC} (and Theorem \ref{northcott} below) for the case of abelian varieties. 
If $h(A)\geq 0$, we can apply Theorem \ref{finiteness main} to finish the proof. 
If $h(A)= 0$, then $A$ is isotrivial, and we can assume that $A$ is constant by a finite extension. Then the whole $\Sha(A)$ is finite by 
Milne \cite{Mil1}.

For (3), by the above argument, we can assume that $A$ has everywhere semi-abelian reduction over $S$. Then the Hodge bundle $\overline\Omega_{A}$ is ample by
R\"ossler \cite[Thm. 1.2]{Ros1}.

Now we prove (2). Let $A$ be as in Corollary \ref{finiteness cor}. The goal is to prove that 
$\Sha(A)[F^\infty]$ is finite. 
By Theorem \ref{positivity main}, there is an isogeny 
$f:A\to A'$ with $A'=B\times_{K} C_{K}$, where $C$ is an abelian variety over $k$, and $B$ is an abelian variety over $K$ with an ample Hodge bundle over $\PP_{k}^1$.

By Theorem \ref{finiteness main},  $\Sha(C_K)[F^\infty]=0$ and 
$\Sha(B)[F^\infty]$ has a finite exponent. Then $\Sha(A')[F^\infty]$ is annihilated by $p^{n_0}$ for some $n_0$. 
Taking Galois cohomology on the exact sequence 
$$
0\lra \ker(f)[K^s] \lra A(K^s) \lra A'(K^s) \lra 0,
$$
we see that the kernel of $\Sha(A)[F^\infty]\to \Sha(A')[F^\infty]$ 
is annihilated by $\deg(f)$.
Thus $\Sha(A)[F^\infty]$ is annihilated by $p^{n_0}\deg(f)$. 
It is finite by Milne \cite{Mil2} again. 
This proves Corollary \ref{finiteness cor}.

\subsubsection*{Map of differentials}

In the following, we prove Theorem \ref{finiteness main}. 
Our proof is inspired by an idea of R\"ossler \cite{Ros1}, which in turn comes from an idea of Kim \cite{Kim}. 
We refer back to \S \ref{section idea} for a quick idea of our proof.

We first introduce some common notations for part (1) and (2). 
Fix an element $X\in \Sha(A)[F^\infty]$. Then $X\in \Sha(A)[F^n]$ for some $n\geq 1$. We need to bound $n$ to some extent. 
Denote $K_n=K^{\frac{1}{p^n}}$, viewed as an extension of $K$.
By Lemma \ref{sha base change}, the base change $X_{K_n}$ is a trivial  $A_{K_n}$-torsor. 
Therefore, the set $X(K_n)$ is non-empty.

Take a point of $X(K_n)$, which gives a closed point $P$ of $X$.
%whose residue field $k(P)$ is isomorphic to 
%$K_{n'}$ for some $n'\leq n$. 
Denote by $\CX$ the N\'eron model of $X$ over $S$. 
Let $\CP_0$ be the Zariski closure of $P$ in $\CX$. Let $\CP$ be the normalization of $\CP_0$. 
By definition, $\CP$ and $\CP_0$ are integral curves over $k$, endowed with quasi-finite morphisms to $S$.

If $X$ is non-trivial in $\Sha(A)$, then $P$ is not a rational point over $K$. 
It follows that the morphism $\psi:\CP\to S$ is a non-trivial purely inseparable quasi-finite morphism over $k$.
We are going to bound the degree of this morphism. 

Start with the canonical surjection
$$
\tau_0: (\Omega_{\CX/S}^1) |_{\CP_0} \lra \Omega_{\CP_0/S}^1.
$$
As $\CP_0$ is purely inseparable over $S$,  we have a canonical isomorphism 
$$
\Omega_{\CP_0/k}^1\lra \Omega_{\CP_0/S}^1.
$$ 
Then we rewrite $\tau_0$ as 
$$
\tau_0: (\Omega_{\CX/S}^1) |_{\CP_0} \lra \Omega_{\CP_0/k}^1.
$$
By pull-back to the normalization $\CP\to \CP_0$, we obtain a nonzero morphism
$$
\tau: (\Omega_{\CX/S}^1) |_{\CP} \lra \Omega_{\CP/k}^1.
$$
Here the restrictions to $\CP$ really mean pull-backs, since $\CP\to \CX$ may not be an immersion but a quasi-finite morphism. 

Denote by $\psi:\CP\to S$ the natural morphism. 
By Lemma \ref{hodge bundle}, we have a canonical isomorphism
$$
(\Omega_{\CX/S}^1) |_{\CP} \lra \psi^*\overline\Omega_{\CA/S}.
$$
Therefore, the nonzero map $\tau$ becomes
$$
\tau: \psi^*\overline\Omega_{\CA/S} \lra \Omega_{\CP/k}^1.
$$
It is a morphism of vector bundles on $\CP$.

\subsubsection*{Proof of part (1)}

With the above map $\tau$, it is very easy to prove part (1).  

In fact, by the assumption in (1), $\CA$ is an abelian scheme over $S$, so $\CX$ is proper and smooth over $S$.
Then $\CP_0$ and $\CP$ are proper curves over $k$.
In particular, $\CP$ is a proper and regular curve over (the perfect field) $k$ with a finite, flat and radicial morphism $\psi$ to $S=\PP_k^1$. Thus $\CP$ is isomorphic to $\PP_k^1$, under which $\psi$ becomes a relative Frobenius morphism.

The nonzero map $\tau$ gives
$$
\mu_{\min}\left(\psi^*\overline\Omega_{\CA/S}\right) \leq 
\deg\left(\Omega_{\CP/k}^1\right)=-2.
$$
By assumption, $\overline\Omega_{\CA/S}$ is nef, so the left-hand side is non-negative. This is a contradiction, which is originally caused by the assumption that $X$ is non-trivial.  
Part (1) is proved.

\subsubsection*{Proof of part (2)}

For part (2) of Theorem \ref{finiteness main}, we do not have the assumption that $\CA\to S$ is proper,  and thus we lose the properness of $\CP_0$ and its normalization $\CP$. 
To resolve the problem, we use a result of R\"ossler \cite{Ros1} to ``compactify'' $\tau$, which is in turn a consequence of the degeneration theory of Faltings--Chai \cite{FC}.

Resume the above notations. We still have a nonzero map 
$$
\tau: \psi^*\overline\Omega_{\CA/S} \lra \Omega_{\CP/k}^1.
$$
Here $\CP$ is still a smooth curve over $k$. 
Denote by $\CP^c$ the unique smooth compactification of $\CP$ over $k$. 
We obtain a finite, flat and radicial morphism $\psi^c:  \CP^c\to S$.  
This is still a relative Frobenius morphism. 

Denote by $E_0$ the reduced closed subscheme of $S$ consisting of $v\in S$ such that $\CA$ is not proper above $v$. Denote by $E$ the reduced structure of the preimage of $E_0$ under the map $\CP^c\to S$.  
We have the following extension. 

\begin{pro} \label{extension}
The map $\tau:\psi^*\overline\Omega_{\CA/S} \lra \Omega_{\CP/k}^1$ extends uniquely to a nonzero map
$$
\tau^c: (\psi^c)^*\overline\Omega_{\CA/S} \lra \Omega_{\CP^c/k}^1(E).
$$
\end{pro}
The uniqueness of the extension is trivial. 
It is easy to see how the proposition finishes proving part (2) of Theorem \ref{finiteness main}. 
In fact, the existence of the map $\tau^c$ gives
$$
\mu_{\min}\left((\psi^c)^*\overline\Omega_{\CA/S}\right) \leq 
\deg\left(\Omega_{\CP^c/k}^1(E)\right).
$$
This is just 
$$
\deg(\psi^c)\cdot \mu_{\min}(\overline\Omega_{\CA/S})  
\leq 2g-2+\deg(E_0)-2.
$$
Here $g$ is the genus of $S$.
Note that $\deg(\psi^c)=[K(P):K]$.
It follows that $K(P)$ is contained in $K_{n_0}$, where $n_0$ is the largest integer satisfying 
$$
p^{n_0}\cdot \mu_{\min}(\overline\Omega_{\CA/S})  
\leq 2g-2+\deg(E_0)-2.
$$
This gives (2) of the theorem.

\subsubsection*{Proof of the extension}

Now we prove Proposition \ref{extension}. 
Note that the map can be extended as
$$
(\psi^c)^*\overline\Omega_{\CA/S} \lra \Omega_{\CP^c/k}^1(mE)
$$
for sufficiently large integers $m$.
The multiplicity $m$ represents the order of poles allowed along $E$, and 
the case $m=1$ is exactly the case of log-differentials.
Our proof takes a lot of steps of reductions.

To control the poles, it suffices to verify the result locally; i.e., we can replace $S$ by its completion at a point in $E$, and replace everything else in the maps by its corresponding base change. To avoid overwhelming notations, we will still use the original notations, but note that we will be in the local situation. As a consequence (of assuming this local situation), $X$ is a trivial torsor, so we assume that $X=A$ and $\CX=\CA$. 
Since we have the trivial torsor, our situation is very similar to the situation of R\"ossler \cite{Ros1}.

\begin{lem}[Rossler]\label{rossler}
Assume further that $A$ has a principal polarization and $A(\overline K)[n]\subset A(K)$ for some $n>2$ coprime to $p$. 
Then the map $\gamma$ extends to a map
$$
\gamma^c: (\psi^c)^*\overline\Omega_{\CA} \lra \Omega_{\CP^c/k}^1(E).
$$
\end{lem}
\begin{proof}
This is essentially \cite[Lem. 2.1]{Ros1}, except that we are in the local case, but it does not make any essential difference in the proof. 
The extension is obtained by applying the compactification result of \cite{FC}.
For convenience of readers, we sketch the proof here. 

Denote $U=S-E_0$. Then $\CA$ is proper over $U$. 
Since $S$ is the spectrum of a complete discrete valuation ring by our assumption, the essential case is that $E_0$ is the closed point of $S$ and $U$ is the generic point of $S$. 
The key is that the abelian scheme $\pi_U:\CA_U\to U$ has a compactification over $S$, which consists of a regular integral scheme $\CV$ containing $\CA_U$ as an open subscheme and a proper morphism $\bar\pi:\CV\to S$ extending 
$\pi_U:\CA_U\to U$. 
The complement $D=\CV-\CA_U$ is a divisor with normal crossings with respect to $k$. 
Moreover, the log-differential sheaf 
$$\Omega_{\CV/S}^1(\log D/E_0):=\Omega_{\CV/k}^1(\log D)/\bar\pi^*
\Omega_{S/k}^1(\log E_0)$$
is locally free on $\CV$ and satisfies
$$\Omega_{\CV/S}^1(\log D/E_0)=\bar\pi^* \overline\Omega_{\CA}, \qquad
\bar\pi_*\Omega_{\CV/S}^1(\log D/E_0)=\overline\Omega_{\CA}. 
$$
This result is a consequence of \cite[Chap. VI, Thm. 1.1]{FC}, which actually constructs a compactification $\bar A_{g,N}$ of the moduli space $A_{g,N}$ of principally polarized abelian varieties of dimensions $g$ with full $N$-level structures and its universal abelian variety. 
The pull-back of the compactification of the universal abelian variety via the map $S\to \bar A_{g,N}$ (representing the family $\CA\to S$) gives the compactification $\CV$ in our notation. 

With the compactification, take $\CR$ to be the closure of $P$ in $\CV$. 
There is a natural finite map $\delta:\CP^c\to \CR$, which is just the normalization of $\CR$. 
Then we have well-defined maps 
$$
\delta^*(\Omega_{\CV/S}^1|_\CR) \lra  \delta^*\Omega_{\CR/k}^1 \lra \Omega_{\CP^c/k}^1. 
$$
Note that the pull-back of a log-differential is still a log-differential.  
The log-version of the above composition give a map
$$
\delta^*(\Omega_{\CV/S}^1(\log D/E_0)|_\CR) \lra \Omega_{\CP^c/k}^1(\log E). 
$$
By the above property of $\Omega_{\CV/S}^1(\log D/E_0)$, it becomes
$$
(\psi^c)^*\overline\Omega_{\CA} \lra \Omega_{\CP^c/k}^1(\log E). 
$$
This is exactly the extension we want. 
\end{proof}

\subsubsection*{Polarization and level structure}

Go back to the proof of Proposition \ref{extension}. It remains to add a polarization and a level structure to $A$. 

We first take care of the polarization. 
By Zarhin's trick, $A^*=(A\times A^t)^4$ has a principal polarization (cf. \cite{Zar} or \cite[IX, Lem. 1.1]{MB2}). 
Write $A^*=A\times A^3\times (A^t)^4$. 
Extend the closed point $P\in A$ to be the point $P^*=(P, 0^3, 0^4)$ in 
$A^*$. Note that 
$\overline \Omega_{A^*}=\overline \Omega_{A}\oplus (\overline \Omega_{A})^3\oplus (\overline \Omega_{A^t})^4$,
and that $A^*$ has the same set of places of bad reduction as $A$. 
The solution of the analogous problem for the version $(A^*, P^*)$ implies that of $(A, P)$. 
Hence, we can assume that $A$ is principally polarized.

In order to get a level structure, we need a descent argument.  
Let $S'\to S$ be a finite, flat and tamely ramified Galois morphism. 
Take this morphism to do a base change, and denote by $(S', \CP', \psi')$ the base changes of $(S, \CP, \psi)$.
Denote by $E'$ the reduced structure of the preimage of $E$ in $\CP'^c$, which is just a point in the local setting. 
Suppose that we have a well-defined extension over $S'$ of the corresponding map $\gamma'$, which should take the form
$$
\gamma'^c: (\psi'^c)^*\overline\Omega_{\CA} \lra \Omega_{\CP'^c/k}^1(E').
$$
Note that the pull-back of $\Omega_{\CP^c/k}^1(E)$ to
$\CP'^c$ is exactly $\Omega_{\CP'^c/k}^1(E')$ by considering the ramification index. 
Taking the Galois invariants on both sides of $\gamma'^c$, we get exactly the desired map $\gamma^c$ on $\CP^c$. 

Finally, we can put a level structure on $A$. 
Take a prime $\ell\nmid (2p)$. 
Let $K'=K(A[\ell])$ be the field of definition of all $\ell$-torsions of $A$. Let $S'$ be the integral closure of $S$ in $K'$. 
We are going to take the base change $S'\to S$.
The only thing left to check is that $K'$ is tamely ramified over $K$.
This is a well-known result proved by Grothendieck under the conditions that 
$p\neq \ell$ and $A$ has semi-abelian reduction. 
In fact, the wild inertia group 
$I^{\rm w}\subset \Gal(K^{\rm sep}/K)$ is a pro-$p$ group. 
By \cite[Exp. IX, Prop. 3.5.2]{SGA7}, the action of $I^{\rm w}$ on the Tate module $T_\ell(A)$ is trivial. In other words, any point of $A(K^{\rm sep})[\ell^\infty]$ is defined over the maximal tamely ramified extension of $K$. 
This finishes the proof of Proposition \ref{extension}.

\section{Reduction of the Tate conjecture} \label{section reduction}

The goal of this section is to prove Theorem \ref{reduction main}. The idea of the proof is sketched in \S \ref{section idea}.
In \S \ref{section prelim tate}, we introduce some preliminary results to be used later. 
In \S \ref{subsection reduction}, we prove Theorem \ref{reduction main}.

\subsection{Preliminary results}  
\label{section prelim tate}

The goal of this subsection is to review some basics of the BSD conjecture, and introduce its equivalence with the Tate conjecture as in the work of Artin--Tate. 
We also introduce a result about projective regular models of abelian varieties as a consequence of the work of Mumford, Faltings--Chai and Kunn\"emann.

\subsubsection*{The BSD conjecture}

The prestigious Birch and Swinnerton-Dyer Conjecture over global fields is as follows:

\begin{conj}[BSD Conjecture:$\, BSD(A)$] \label{Tate conjecture}
Let $A$ be an abelian variety over a global field $K$. Then 
$$\ord_{s=1}L(A,s)=\rank\, A(K). $$
\end{conj}

Recall that the global L-function
$$
L(A,s)=\prod_v L_v(A,s)
$$
is the product over all non-archimedean places $v$ of $K$, where the local L-function
$$
L_v(A,s)=\det(1-q_v^{-s}\mathrm{Frob}(v)|V_\ell(A)^{I_v})^{-1}.
$$
Here $q_v$ is the order of the residue field of $v$, 
$\mathrm{Frob}(v)$ is a Frobenius element of $v$ in $\Gal(K^s/K)$, 
$I_v$ is the Inertia subgroup of $v$ in $\Gal(K^s/K)$, 
$\ell$ is any prime number different from the residue characteristic of $v$,
and
$V_\ell(A)$ is the $\ell$-adic Tate module of $A$.

In this paper, we are only interested in the case that $K$ is a global function field. 
In this case, $L(A,s)$ is known to be a rational function of $q^{-s}$, where $q$ is the order of the largest finite field contained in $K$. See \cite[VI, Example 13.6(a)]{Mil4} for example.
The abelian group $A(K)$ is finitely generated by the Lang--N\'eron theorem, as in \cite[Thm. 2.1]{Con}.
Moreover, in this case, we always know
$$\ord_{s=1}L(A,s)\geq \rank\, A(K)$$
by the works \cite{Tat3, Bau}, as a consequence by the comparison with the Tate conjecture which will be reviewed below. 

We will need the following easy results, which can be checked by treating both sides of the BSD conjecture. 

\begin{lem} \label{isogeny}
\begin{enumerate}[(1)]
\item Let $A$ and $B$ be isogenous abelian varieties over a global function field. Then the BSD conjecture holds for $A$ if and only if the BSD conjecture holds for $B$. 
\item Let $A$ and $B$ be any abelian varieties over a global function field. Then the BSD conjecture holds for $A\times B$ if and only if the BSD conjecture holds for both $A$ and $B$. 
\end{enumerate}
\end{lem}

\subsubsection*{Tate conjecture vs BSD conjecture}

The bridge between the Tate conjecture and the BSD conjecture is via fibrations of surfaces. 
Recall that the Tate conjecture $T^1(X)$ (cf. Conjecture \ref{Tate conjecture}) for a projective and smooth surface $X$ over a finite field $k$ asserts that
 for any prime $\ell\neq p$, the cycle class map 
$$
\Pic(X)\otimes_\ZZ \QQ_p \lra H^2(X_{\bar k}, \QQ_\ell(1))^{\Gal(\bar k/k)}
$$
is surjective.

By a \emph{fibered surface} over a field $k$, we mean a projective and flat morphism $\pi:X\to S$, where $S$ is a projective and smooth curve over $k$ and $X$ is a projective and smooth surface over $k$, such that the generic fiber of $X\to S$ is smooth. 
%The fibered surface is called a \emph{Lefschetz fibered surface} if every fiber of $\pi:X\to S$ is semistable with at most one singular point.

Then we have the following beautiful result of Artin--Tate. 

\begin{thm}[Artin--Tate] \label{Brauer Sha}
Let $\pi:X\to S$ be a fibered surface over a finite field $k$. Denote by $J$ the Jacobian variety of the generic fiber of $\pi$. 
Then $T^1(X)$ is equivalent to $BSD(J)$. 
\end{thm}

This equivalence is a part of \cite[\S 4, (d)]{Tat3}, which actually treats equivalence of the refined forms of the conjectures.
See also \cite{Ulm} for a nice exposition of the theorem. 
For further results related to this equivalence, including results about the Tate--Shafarevich group and the Brauer group,
we refer to \cite{Tat3, Mil3, Bau, Sch, KT}.

For a projective and smooth surface $X$, to convert it to a fibered surface, one usually needs to blow-up $X$ along a smooth center. 
The following result asserts that this process does not change the Tate conjecture. 

\begin{lem} \label{birational}
Let $X'\to X$ be a birational morphism of projective and smooth surfaces over a finite field. Then $T^1(X)$ is equivalent to $T^1(X')$.
\end{lem}

This can be checked by directly describing the change of both sides of the conjectures.

\subsubsection*{Projective regular integral models of abelian varieties}

The following results asserts that we have well-behaved regular projective 
models of abelian varieties with semi-abelian reduction. 

\begin{thm} \label{model}
Let $S$ be a connected Dedekind scheme with generic point $\eta$, and let $A$ be an abelian variety over $\eta$ with semi-abelian reduction over 
$S$. Then there is a {projective, flat and regular} integral model 
$\psi:\CP\to S$ of $A$ over $S$ such that 
 there is a canonical $\CO_S$-linear isomorphism 
$$
R^1\psi_* \CO_\CP \lra \Lie(\CA^\vee/S).
$$
Here $\CA^\vee$ is the N\'eron model over $S$ of the dual abelian variety $A^\vee$ of $A$.
\end{thm}
\begin{proof}
This follows from the theory of degeneration of abelian varieties of Mumford \cite{Mum1} and Faltings--Chai \cite{FC}. 
In particular, by the exposition of K\"unnemann \cite{Kun}, the degeneration theory gives an explicit compactification of a semi-abelian scheme from a reasonable rational polyhedral cone decomposition. 
For the purpose of our theorem, choose $\CP$ to be the integral model constructed in \cite[Thm. 4.2]{Kun}. 
We claim that it automatically satisfies the property of the cohomology. 
Note that we have a canonical isomorphism $H^1(A, \CO_A)\to \Lie(A^\vee/\eta)$, as expressions of the tangent space of the Picard functor $\underline{\Pic}_{A/\eta}$. 
Then it remains to prove that this isomorphism extends to the integral version over $S$. 
This essentially follows from the special case $(s,a,b)=(1,1,0)$ of \cite[Chap. VI, Thm. 1.1(iv)]{FC}, which is proved in \S VI.2 of loc. cit..
We can check literally their proof works in our case. 
Alternatively, we introduce a different approach in the following. 

First, the truth of our isomorphism does not depend on the choice of the rational polyhedral cone decomposition, as mentioned at the beginning of page 209 in loc. cit..
Second, the isomorphism 
$R^1\bar f_* \CO_{\bar Y} \to \Lie(G/\bar X)$
of \cite[Chap. VI, Thm. 1.1(iv)]{FC} is compatible with base change by any morphism $Z\to \bar X$.
In other words, the map $\bar f$ is cohomologically flat in dimension 1.  
In fact, by the semi-continuity theorem, this holds if $h^1(Y_s, \CO_{Y_s})$ is constant in $s\in \bar X$, which can be seen from their proof. 
Once we have the cohomological flatness, our result holds if $A$ is principally polarized. 
In fact, take a level structure by extending $S$ if necessary, 
and then we have a map $S\to \bar X$ by the moduli property. 
Then the pull-back of $R^1\bar f_* \CO_{\bar Y} \to \Lie(G/\bar X)$ to $S$ gives the isomorphism we need.
Finally, if $A$ does not have a principal polarization, we can apply Zarhin's trick as in our treatment of Proposition \ref{extension}. 
\end{proof}

\subsection{Reduction of the Tate conjecture} \label{subsection reduction}

Now we prove Theorem \ref{reduction main}.
Let $X$ be a projective and smooth surface over $k$.
We will convert $T^1(X)$ to $T^1(\CY)$ for some projective and smooth surface $\CY$ over $k$ with $H^1(\CY, \CO_\CY)=0$.

\subsubsection*{Step 1: Make a fibration}
By Nguyen \cite{Ngu}, there is a Lefschetz pencil in $X$ over the finite field $k$.
This is a version over finite field of the existence of Lefschetz pencils in \cite[Exp. XVII, \S 3]{SGA7}. 
Blowing-up $X$ along the base locus of the Lefschetz pencil, we get a birational morphism $X'\to X$ and a fibered surface $\pi:X'\to S$ with $S=\PP^1_k$.
Here $X'$ is smooth over $k$ as the base locus is reduced.  
Denote by $J$ the Jacobian variety of the generic fiber of $\pi:X'\to S$, which is an abelian variety over $K=k(t)$. 

Since $\pi$ is semistable, $J$ has semi-abelian reduction over $S=\PP_k^1$.
In fact, by \cite[\S9.5, Thm. 4(b)]{BLR}, the Picard functor 
$\underline{\Pic}_{X'/S}^0$ is isomorphic to the relative identity component of the N\'eron model of $J$.
By \cite[\S9.2, Prop. 10]{BLR}, $\underline{\Pic}_{X'_s/s}^0$ is semi-abelian for any closed point $s\in S$.

By Lemma \ref{birational}, $T^1(X)$ is equivalent to $T^1(X')$. 
By Theorem \ref{Brauer Sha}, $T^1(X')$ is equivalent to $BSD(J)$. 

\subsubsection*{Step 2: Make the Hodge bundle positive}

We will prove that $BSD(J)$ is equivalent to $BSD(A)$ for an abelian variety $A$ over $K$ with everywhere semi-abelian reduction and with an ample Hodge bundle over $S$.

Apply Theorem \ref{positivity main} to $J$. 
Then $J$ is isogenous to 
$A\times_{K} C_{K}$, where $C$ is an abelian variety over $k$, and $A$ is an abelian variety over $K$ with an ample Hodge bundle over 
$S$.
Note that $A$ also has semi-abelian reduction by \cite[\S7.3, Cor. 7]{BLR}.  
By Lemma \ref{isogeny}, $BSD(J)$ is equivalent to the simultaneous truth of $BSD(A)$ and $BSD(C_K)$.

By \cite{Mil1}, $BSD(C_K)$ holds unconditionally. 
Alternatively, in the current case of $K=k(t)$, it is easy to prove that both sides of the BSD conjecture is 0. 
For the Mordell--Weil rank, we have
$$
C_K(K)=\Hom_S(S,C_S)=\Hom_k(S,C)=C(k)
$$
is finite. For the L-function, one can also have an explicit expression in terms of the eigenvalues of the Frobenius acting on the Tate module of 
$C$.

Therefore, $BSD(J)$ is equivalent to $BSD(A)$.

\subsubsection*{Step 3: Take projective regular model}

Let $\psi:\CP\to S$ be a projective, flat and regular integral model of $A^\vee$ over $S$ as in Theorem \ref{model}. 
In particular, we have a canonical isomorphism  
$$
R^1\psi_* \CO_\CP\lra \Lie(\CA/S).
$$
Here $\CA$ is the N\'eron model of $A$ over $S$.
Then the dual of  $R^1\psi_* \CO_\CP$ is isomorphic to the Hodge bundle of $A$, which is ample by construction. 

By the Leray spectral sequence for $\psi:\CP\to S$, we have an exact sequence
$$
0\lra H^1(S, \CO_S) \lra H^1(\CP, \CO_\CP) \lra 
H^0(S, R^1\psi_* \CO_\CP)\lra 0.
$$
The term $H^0(S, R^1\psi_* \CO_\CP)$ vanishes by the ampleness of the dual of 
$R^1\psi_* \CO_\CP$. Therefore, we end up with 
$H^1(\CP, \CO_\CP)=0.$

\subsubsection*{Step 4: Take a surface in the regular model}
Note that $\CP$ is a projective and smooth variety over $k$ with $H^1(\CP, \CO_\CP)=0$. 
We claim that there is a projective and smooth $k$-surface $\CY$ in $\CP$ satisfying the following conditions:
\begin{enumerate}[(1)]
\item $H^1(\CY, \CO_\CY)=0$.
\item The canonical map 
$H^1(\CP_\eta, \CO_{\CP_\eta})\to H^1(\CY_\eta, \CO_{\CY_\eta})$ is injective.
\item The generic fiber $\CY_\eta$ of $\CY\to S$ is smooth. 
\end{enumerate}
Here $\eta$ is the generic point of $S$.

This is a consequence of the Bertini-type theorem of Poonen \cite{Poo}. 
By induction on the codimension of $\CY$ in $\CP$, it suffices to prove that there is a smooth hyperplane section $\CY$ of $\CP$ satisfying (3), since (1) and (2) are automatic.
For example, (1) follows from the vanishing of $H^2(\CP,\CO(-\CY))$, which holds if $\CY$ is sufficiently ample. 
To achieve (3), it suffices to make the closed fiber $\CY_s$ smooth over $s$ for some closed point $s\in S$ such that $\CP_s$ is smooth. 
Take a very ample line bundle $\CL$ over $\CP$ such that $H^0(\CP,\CL)\to H^0(\CP_s,\CL_s)$ is surjective.
The complete linear series of $\CL$ defines a closed immersion $\CP\to \PP_k^N$. 
Denote $\Sigma_d=H^0(\PP_k^N, \CO_{\PP_k^N}(d))$, and denote by $\Sigma$ the disjoint union of $\Sigma_d$ for all $d\geq1$.
Denote $m=\dim \CP$. 
By Poonen \cite[Thm 1.1]{Poo}, we have the following results:
\begin{enumerate}[(a)]
\item The density of $f\in \Sigma$ such that $\div(f)\cap \CP$ is smooth over $k$ is $\zeta_{\CP}(m+1)^{-1}$. 
\item The density of $f\in \Sigma$ such that $\div(f)\cap \CP_s$ is smooth over $s$ is $\zeta_{\CP_s}(m)^{-1}$. 
\end{enumerate}
We claim that $\zeta_{\CP_s}(m)$ goes to 1 as $[k(s):s]$ goes to infinity. 
In fact, this is easily seen by the Riemann hypothesis proved by Weil. 
As a consequence, we can choose $s\in S$ such that  
$\zeta_{\CP}(m+1)^{-1}+\zeta_{\CP_s}(m)^{-1}>1$. 
Consequently, we can find $f\in \Sigma$ simultaneously satisfying (a) and (b).
Then $\CY=\div(f)\cap \CP$ satisfies condition (3).  
This proves the existence of $\CY$. 

Let $\CY$ be a surface in $\CP$ with the above properties. Denote by $B$ the Jacobian variety of $\CY_\eta$ over $\eta$. 
Consider the homomorphism 
$A\to B$ induced by the natural homomorphism 
$\underline{\Pic}_{\CP_\eta/\eta}\to \underline{\Pic}_{\CY_\eta/\eta}$.
The induced map between the Lie algebras is exactly the injection in (2). 
Therefore, the kernel of $A\to B$ is finite. 
It follows that $A$ is a direct factor of $B$ up to isogeny. 
By Lemma \ref{isogeny}, $BSD(A)$ is implied by
$BSD(B)$.  
By Theorem \ref{Brauer Sha}, $BSD(B)$ is equivalent to $T^1(\CY)$. 

In summary, $T^1(X)$ is implied by $T^1(\CY)$. 
By construction, $\CY$ is a projective and smooth surface over $k$ with 
$H^1(\CY, \CO_\CY)=0$.
This finishes the proof of Theorem \ref{reduction main}.

%------------------------------------------------------
\end{document}